\documentclass[11pt, reqno]{amsart}
\usepackage{leading}
\usepackage{blindtext}
\usepackage[new]{old-arrows}
	
\makeatletter
\DeclareMathAlphabet{\mathpzc}{OT1}{pzc}{m}{it}
\usepackage[margin=2.4cm]{geometry}
\usepackage{amsmath, amssymb, amsfonts, amstext, verbatim, amsthm, mathrsfs}
\usepackage[mathcal]{eucal}
\usepackage{microtype}
\usepackage[colorlinks=true,linkcolor=blue,citecolor=red,urlcolor=blue,citebordercolor={0 0 1},urlbordercolor={0 0 1},linkbordercolor={0 0 1}]{hyperref} 
\usepackage{tikz-cd}
\def\makeCal#1{
	\expandafter\newcommand\csname c#1\endcsname{\mathcal{#1}}}
\def\makeBB#1{
	\expandafter\newcommand\csname b#1\endcsname{\mathbb{#1}}}
\def\makeFrak#1{
	\expandafter\newcommand\csname f#1\endcsname{\mathfrak{#1}}}

\count@=0
\loop
\advance\count@ 1
\edef\y{\@Alph\count@}
\expandafter\makeCal\y
\expandafter\makeBB\y
\expandafter\makeFrak\y
\ifnum\count@<26
\repeat

\newtheorem{theorem}{Theorem}[section]
\newtheorem{proposition}[theorem]{Proposition}
\newtheorem{lemma}[theorem]{Lemma}
\newtheorem{corollary}[theorem]{Corollary}

\theoremstyle{definition}
\newtheorem{definition}[theorem]{Definition}
\newtheorem{remark}[theorem]{Remark}
\newtheorem{example}[theorem]{Example}
\newtheorem{question}[theorem]{Question}

\newcommand {\id}{\operatorname{id}}

\newcommand{\coker}{\operatorname{coker}}

\newcommand{\End}{\operatorname{End}}
\newcommand{\spec}{\operatorname{spec}}

\newcommand{\ul}{Ulrich }
\newcommand{\rul}{relatively Ulrich }
\newcommand{\ses}{short exact sequence }
\newcommand{\les}{long exact sequence }
\newcommand{\wrt}{with respect to }

\numberwithin{equation}{section}
\usepackage{expl3}

\begin{document}

\title[Triviality of direct image of coherent sheaves]{On the triviality of direct image of coherent sheaves}

\author[I. Biswas]{Indranil Biswas}

\address{Department of Mathematics, Shiv Nadar University, NH91, Tehsil
Dadri, Greater Noida, Uttar Pradesh 201314, India}

\email{indranil.biswas@snu.edu.in, indranil29@gmail.com}

\author[J. Pine]{Jagadish Pine}

\address{Indian Institute of Science Education and Research, Pune,  Main Academic, Dr Homi Bhabha Rd, Pashan, Maharashtra 411008}

\email{math.jagadishpine@gmail.com, jagadish.pine@iiserpune.ac.in}

\date{}
	
	

\begin{abstract}
Let $\pi\,:\, X \,\longrightarrow\, Y$ be a finite morphism of projective varieties
defined over an algebraically closed field of characteristic zero. We study the
necessary and sufficient criteria for $\pi$ such that there exists a coherent sheaf $E$ on $X$ whose direct
image $\pi_*E$ is a trivial vector bundle on $Y$ of positive rank. When $X$ is smooth and $Y$ is Cohen-Macaulay, such a coherent sheaf is necessarily locally free. We show that the existence of such a coherent sheaf $E$ is guided by the properties of the branching divisor of $\pi$. When the covering $\pi\,:\, X \,\longrightarrow\, Y$ is
admissible abelian Galois, we give a complete answer. As an application, it is shown that
every smooth admissible abelian Galois covering of $\mathbb{P}^n$ supports an Ulrich bundle.
\end{abstract}
	
\maketitle
	
\section{Introduction}

Throughout the article, the base field is denoted by $k$, which is assumed to be algebraically closed of characteristic
zero. Throughout, by a variety we mean an integral separated scheme of finite type over $k$. The purpose of this article is to investigate the following question:
	
\begin{question} \label{mainquest}
Let $\pi\,:\, X \,\longrightarrow\, Y$ be a finite morphism of projective varieties. Under what conditions does
there exist a coherent sheaf $E$ on $X$ such that the direct image $\pi_* E$ is a trivial vector bundle of positive rank on $Y$?
\end{question}
	
In this paper, we analyze both necessary and sufficient conditions on the finite map $\pi$ for the 
existence of such a coherent sheaf $E$ on $X$.

The motivation for Question \ref{mainquest} came from the study of \ul vector bundles in smooth projective varieties. 
Consider a polarized variety $(X,\, \cO_X(1))$ with an embedding $X\,\hookrightarrow \,\bP^N$ given by the 
very ample line bundle $\cO_X(1)$. Denote the dimension of $X$ by $n$. Consider a finite linear projection 
$\pi\,:\, X \,\longrightarrow \,\bP^n$ such that $\pi^*\cO_{\bP^n}(1)\,\cong\,\cO_X(1)$. Then a vector bundle
$E$ on $X$ is called \ul \wrt $\pi$
if the direct image $\pi_*E$ is trivial. When $\cO_X(1)$ is an ample and globally generated 
polarization, this definition of \ul vector bundle has a natural extension due to the following result 
\cite[Proposition 5.4]{eisenbud2003resultants}: If there exists an \ul vector bundle on $X$ for a polarization 
$H$, then there is an \ul vector bundle of much larger rank for any multiple $\alpha \cdot H$, where $\alpha
\,\in\, \bZ_{>0}$.

The question \ref{mainquest} is a natural generalization of this to the above setup of a given finite morphism 
$\pi\,:\, X \,\longrightarrow\, Y$ between projective varieties. Building on the ideas introduced in \cite{MOHANKUMAR2025107946}, the notion of a \emph{relatively Ulrich bundle} was defined in \cite[Definition 4.1]{parameswaran2024ulrich}. Very recently, we became aware that this notion had also been defined more generally for coherent sheaves in \cite[Definition 4.14]{AK2026}, although for quite different reasons. Here, we state this notion in this more general form, namely for coherent sheaves. The definition is recalled below.
	
\begin{definition} \label{defi}
Let $\pi\,:\, X \,\longrightarrow \, Y$ be a finite morphism between projective varieties.
Then a coherent sheaf $E$ on $X$ is called a relatively \ul sheaf \wrt $\pi$ if $\pi_*E \cong \cO_{Y}^{\oplus r}$ for some natural number $r \in \bN$. 
\end{definition}
	
\begin{remark} \label{bundleremark1}
	By Proposition \ref{mostcrucialprop1}, if a \rul sheaf exists on $X$ \wrt $\pi$, then $\pi$ is surjective and $E$ is torsionfree. The same proposition also shows the following: If $X$ is normal, and $Y$ satisfies $S_2$, then $E$ is a reflexive sheaf. Moreover, if $X$ is smooth and $Y$ is Cohen-Macaulay, then every \rul sheaf on $X$ \wrt $\pi$ is locally free.
\end{remark}	
	
Thus, when $Y$ is a projective space, a relatively \ul bundle is an \ul bundle. Take two
finite morphisms $\pi_1\,:\, X\, 
\longrightarrow\, Y$ and $\pi_2\,:\, Y \,\longrightarrow\, \bP^n$, and set
$\pi\,:=\, \pi_2 \circ \pi_1 \,:\, X \,\longrightarrow\,
\bP^n$ to be the composition of them. Let $\cO_Y(1)\,:=\, \pi_2^* \cO_{\bP^n}(1)$ and
$\cO_X(1)\,:=\, \pi^* \cO_{\bP^n}(1)$ be the ample line bundles on $Y$ and $X$ respectively. Let $E$ (respectively, $F$)
be a \rul vector bundle on $X$ (respectively, $Y$) \wrt $\pi_1$ (respectively, $\pi_2$). Then using
the projection formula, we see that $(\pi_1^*F) \otimes E$ is an \ul vector bundle on 
$X$ \wrt $\cO_X(1)$.
	
The study of  \rul vector bundles is primarily motivated by \cite[Theorem 
1.2]{parameswaran2024ulrich}. This theorem of \cite{parameswaran2024ulrich}
reduces the problem of showing the existence of an \ul vector
bundle for a cyclic covering $\pi\,:\, X \,\longrightarrow\, \bP^n$ to the problem of finding a relatively \ul vector bundle on a cyclic covering of 
certain complete intersection curves in $\bP^n$.
Thus, the Eisenbud-Schreyer conjecture, \cite{eisenbud2003resultants}, on the existence of \ul vector bundles is closely related to 
the existence of \rul vector bundles. This approach also suggests a new perspective on the Eisenbud-Schreyer conjecture.

As shown in Lemma \ref{nonexistenceetale} and Corollary \ref{cor1}, nontrivial \'{e}tale covering maps between irreducible smooth projective varieties do not admit any relatively Ulrich vector bundle. Accordingly, we restrict our attention to finite ramified covering maps. In this paper, we address Question \ref{mainquest} for ramified abelian Galois covering maps. The main result is described below.

Let $\pi\,:\, X \,\longrightarrow\, Y$ be an admissible abelian Galois covering of degree $d$
between projective varieties (see Definition \ref{admissibleab}). So it is
a composition of finitely many standard cyclic coverings of degrees $d_1,\, d_2,\, \cdots, \,d_l$; see Section 3. In particular, standard cyclic covers are admissible covers. Then the sheaf 
$\pi_* \cO_X$ splits as a direct sum of line bundles, consisting of the line bundles $\cO_Y$, $\{M_i\,\,\big\vert\,\, 1 
\,\le\, i \,\le\, l\}$, and their tensor products (see \eqref{a3}).

Our main theorem provides a necessary and sufficient condition for the existence of a \rul coherent sheaf on $X$ \wrt an admissible abelian Galois covering map $\pi$.

\begin{theorem}[{Proposition \ref{dfoldsurj}, Theorem \ref{sufficient}}, and Corollary \ref{admissiblecor}]\label{main1}
There exists a \rul coherent sheaf $E$ on $X$ \wrt $\pi$ if and only if the branch divisor
$B_i \,\in\, H^0(Y,\, M_i^{\otimes d_i})$ is in the image of the $d_i$-fold multiplication map
$$H^0(Y,\, M_i)^{\otimes d_i}\ \longrightarrow\ H^0(Y,\, M_i^{\otimes d_i})$$
for all $1 \,\le\, i \,\le\, l$.
\end{theorem}

\begin{remark}
In fact, we prove a sufficient criterion for a more general class of covering maps, namely standard abelian Galois covers; see Corollary \ref{standardsuff}. As observed in Remark \ref{bundleremark1}, when $X$ is smooth, and $Y$ is Cohen-Macaulay,
Theorem \ref{main1} provides a necessary and sufficient criterion for the existence of relatively Ulrich vector bundles on $X$ \wrt $\pi$. 
\end{remark}

Set $Y$ to be $\bP^n$. Then for an admissible abelian Galois covering $\pi$, we see that the line bundles $M_i$ are
very ample. In this case, the $d_i$-fold multiplication map is surjective, and thus the branch divisors $B_i$ 
are in the image of $H^0(Y,\, M_i)^{\otimes d_i}$. Let $\cO_X(1)$ be the ample and globally generated line 
bundle $\pi^* \cO_{\bP^n}(1)$. Thus, as a consequence of Theorem \ref{main1} we get the following:

\begin{corollary} \label{main2}
Let $\pi\,:\, X \,\longrightarrow\, \bP^n$ be an admissible abelian Galois covering. Then $X$ supports
an \ul sheaf \wrt $\cO_X(1)$.
\end{corollary}

Therefore, by Proposition \ref{mostcrucialprop1}, whenever the total space $X$ of the admissible abelian cover is smooth, it supports an \ul vector bundle \wrt $\cO_X(1)$. Then, as a special case of Corollary \ref{main2}, when $\pi$ is a cyclic covering of degree $d\,:=\, d_1$, we 
recover \cite[Theorem 1.1]{parameswaran2024ulrich} and \cite[Theorem 1.2]{MOHANKUMAR2025107946}.

An interesting connection  between relatively Ulrich sheaves and $\bA^1$-homotopy theory is developed in \cite{AK2026}. For a finite and surjective morphism $\pi\,:\,X \longrightarrow Y$ of smooth varieties, there is a certain line bundle $L$ on $X$, called the relative orientation of $\pi$. The authors prove a sufficient criterion for the well definedness of $\bA^1$-degree of $\pi$. In particular,  they show that if $L$ is relatively Ulrich \wrt $\pi$ and $H^0(Y,\, \cO_Y)\,=\,k$, then the $\bA^1$-degree of $\pi$ is well defined \cite[Theorem A]{AK2026}. Consequently, these developments give further significance to our necessary and sufficient criterion for the existence of \rul sheaves.

\begin{remark}
	The $d_i$-fold multiplication map $H^0(Y,\, M_i)^{\otimes d_i}\ \longrightarrow\ H^0(Y,\, M_i^{\otimes d_i})$ factors through the symmetric power $Sym^{d_i} (H^0(Y,\, M_i))$.  Indeed, locally, after choosing a trivialization of $M_i$, sections of $M_i$ may be viewed as regular functions , and multiplication of regular functions is commutative. Hence, for any permutation $\sigma \,\in\, S_{d_i}$, the tensors $s_1 \otimes s_2 \otimes \cdots \otimes s_{d_i}$ and $s_{\sigma(1)} \otimes s_{\sigma(2)} \otimes \cdots \otimes s_{\sigma(d_i)}$ have same image under the multiplication map. Therefore the multiplication map descends to a well defined map $Sym^{d_i} (H^0(Y,\, M_i)) \,\longrightarrow\, H^0(Y,\, M^{\otimes d_i})$.
\end{remark}

The preceding criterion for the existence of \rul sheaves is related to the normal generation of line bundles. In particular, if the line bundles $M_i$  on a variety $Y$ are normally generated, then every admissible abelian cover $\pi\,:\,X \longrightarrow Y$ admits  \rul sheaves. Indeed, normal generation requires the surjectivity of the corresponding multiplication maps in all degrees, whereas the present criterion involves only the fixed degrees $d_i$ relevant to the cover. Along these lines, Theorem \ref{main1} is applied in Example \ref{ngex1} and Example \ref{ngex2} to  construct interesting examples of abelian Galois covers admitting \rul sheaves.

In Corollary \ref{comintcor}, we prove that admissible abelian Galois covers of complete intersection subvarieties of $\bP^N$ admit \rul sheaves. In Proposition \ref{minrank1} and Corollary \ref{minrank2}, we make some observations concerning the possible ranks of  \rul sheaves. In particular, in Corollary \ref{minrank2}, we exhibits cases in which, for branch divisor of certain specific forms, there exist \rul sheaves of rank $1$. Moreover, in Remark \ref{gencliffrel}, we describe a relation between finite dimensional representations of certain generalized Clifford algebras and the existence of \rul sheaves.

For the rest of this section, let $\pi\,:\, X \,\longrightarrow\, Y$ be a
ramified cyclic Galois covering of degree $d$ between smooth
projective varieties. In Proposition \ref{ggcyclic}, it is proved that the existence of a \rul vector
bundle on 
$X$ implies that $M$ (see \ref{a4}) is a globally generated line bundle on $X$. Our sufficient criterion 
for the existence of a \rul vector bundle, namely the condition that the branch divisor is in the image of
the $d$-fold multiplication $H^0(Y,\, M)^{\otimes d} \, \longrightarrow\,
H^0(Y,\, M^{\otimes d})$, is actually stronger than the condition that $M$ is globally 
generated (see Proposition \ref{dfoldgg}). In Example \ref{ggdfoldexample}, we show that these two 
conditions are not equivalent. This provides examples of cyclic coverings $\pi\,:\, X \,\longrightarrow\, Y$ such 
that $M$ is a globally generated line bundle and yet there is no \rul vector bundle on $X$.

\section{Examples of relatively \ul bundles}

The base field is denoted by $k$, which is assumed to be algebraically closed and of characteristic zero. In 
this section, we provide examples of \rul vector bundles. It is also shown that the existence of \rul
vector bundles is connected to the 
size of the ramification divisor of $\pi$. When $\pi$ is \'etale, or when the
ramification divisor for $\pi$ is small, \rul vector bundles don't exist. We also prove a proposition showing that a coherent sheaf whose pushforward under a finite morphism is locally free is necessarily torsionfree. Moreover, when the domain is smooth and the target is Cohen-Macaulay, such a sheaf is locally free.

Let $C$ be an irreducible smooth projective curve over $k$. Denote the genus of $C$ by $g$.
Let
$$
\Theta\ :=\ \{L\, \in\, {\rm Pic}^{g-1}(C)\,\, \big\vert\,\, H^0(C,\, L)\, \not=\, 0\}
$$
be the theta divisor on ${\rm Pic}^{g-1}(C)$.
For a vector bundle $E$ on $\bP^1$, the vector bundle $E\otimes \cO_{\bP^1}(d)$ will be denoted
by $E(d)$.

Let $\pi\,:\,C \,\longrightarrow \,\bP^1$ be a finite ramified covering map of degree $d$.
Since all vector bundles on $\bP^1$ or a rank $r$ decomposes as $\bigoplus_{i=1}^r {\mathcal O}_{\bP^1}
(n_i)$, with $n_i\, \in\, {\mathbb Z}$, a vector bundle $E$ on $\bP^1$ is trivial if and only if
$H^j(\bP^1,\, E(-1))\,=\,0$ for $j\,=\,0,\, 1$. For a line bundle $L \,\in\, 
\text{Pic}(C)$, we have $\pi_* L \,\cong\, \cO_{\bP^1}^{\oplus d}$ if and only if $$H^j(\bP^1,\, 
(\pi_* L)(-1)) \ \cong\ H^j(C,\, L \otimes \pi^*\cO_{\bP^1}(-1))\ =\ 0$$ for $j\,=\,0,\, 1$. So $\pi_* L 
\,\cong\, \cO_{\bP^1}^{\oplus d}$ if and only if $\text{degree}(L\otimes\pi^*\cO_{\bP^1}(-1))\,=\,g-1$ and
$H^0(C,\, L\otimes\pi^*\cO_{\bP^1}(-1))\,=\,0$. In other words, the following holds:

\begin{lemma}
A line bundle $L \,\in\, {\rm Pic}(C)$ is relatively \ul \wrt $\pi$
if and only if $$L\otimes\pi^*\cO_{\bP^1}(-1) \ \in\ {\rm Pic}^{g-1}(C) \setminus \Theta .$$
\end{lemma}

Let $M_C(r, d)$ denote the moduli space of $S$-equivalence classes of semistable vector bundles
on $C$ of rank $r$
and degree $d$. The moduli space $M_C(r, r(g-1))$ has the generalized Theta divisor
$\Theta$ given by the locus of all $E$ with $H^0(C,\, E)\, \not=\, 0$.

Using similar arguments as above, the following is deduced.

\begin{lemma}
A vector bundle of rank $r$ on $C$ is relatively \ul \wrt $\pi$ if and only if
$$E\otimes\pi^*\cO_{\bP^1}(-1)\ \in\ M_C(r, r(g-1)) \setminus \Theta .$$
\end{lemma}

In the following lemma, by a nontrivial \'etale covering $\pi\,:\, C \,\longrightarrow\, D$, we mean
that $C$ is not a disjoint union of finitely many copies of $D$.

\begin{lemma}\label{nonexistenceetale}
Nontrivial \'etale covering maps $\pi\,:\, C \,\longrightarrow\, D$ between smooth
projective curves do not admit any \rul vector bundle.
\end{lemma}

\begin{proof}
For an \'etale covering map $\pi\,:\,C \,\longrightarrow\, D$ with
$\text{degree}(\pi)\,=\,d \,\ge\, 2$, by Riemann--Hurwitz formula,
\begin{equation}\label{e0}
2g(C)-2\ =\ d \cdot (2g(D)-2),
\end{equation}
where $g(D)$ (respectively, $g(C)$) is the genus of $D$ (respectively, $C$).
Let $V$ be a \rul vector bundle of rank $r$ on $C$ \wrt $\pi$. Then
\begin{equation}\label{e1}
H^0(C,\, V)\,\cong\,H^0(D,\, \pi_*V)\,=\, H^0(D,\, \cO_D^{\oplus rd})\,=\, k^{\oplus rd}.
\end{equation}
Using \eqref{e0} it follows that for a vector bundle $F$ on $C$ of rank $r'$,
\begin{equation}\label{e1b}
\text{degree}(F)\ =\ \chi(F)+ r'(g(C)-1) \ =\ \chi(\pi_* F)+ r'd(g(D)-1)\ =\
\text{degree}(\pi_*F).
\end{equation}
In particular, we have
\begin{equation}\label{e1c}
\text{degree}(V)\ =\ \text{degree}(\pi_* V)\ =\ 0.
\end{equation}

It can be shown that the vector bundle $V$ is semistable. Indeed, if $W\, \subset\, V$ is
a subbundle with $\mu(W)\, > \, \mu(V)\,=\, 0$ (see \eqref{e1c}), where $\mu(A)\,:=\,
\frac{\text{degree}(A)}{\text{rank}(A)}$ (see \cite[Ch.~1, \S~1.2]{HL}), then from \eqref{e1b},
$$
\mu(\pi_* W)\ =\ \frac{1}{d}\cdot \mu (W) \ >\ 0.
$$
Now the subbundle $\pi_*W\, \subset\, \pi_*V\,=\, \cO_D^{\oplus rd}$ contradicts that fact that
$\cO_D^{\oplus rd}$ is semistable. Thus the vector bundle $V$ is semistable.

Since $V$ is semistable of degree zero (see \eqref{e1c}), it can be shown that
\begin{equation}\label{e1d}
\dim H^0(C,\, V)\ \leq \ {\rm rank}(V)\ =\ r.
\end{equation}
To see that take any point $y\, \in\, C$, and consider the evaluation map
$\phi_y\, :\, H^0(C,\, V)\,\longrightarrow\, V_y$ defined by $s\, \longmapsto\, s(y)$.
If \eqref{e1d} fails, then there is a nonzero section $0\, \not=\, s_0\, \in\, H^0(C,\, V)$
such that $\psi_y(s_0)\,=\, 0$. Let
$$
\gamma\ :\ {\mathcal O}_C\ \longrightarrow\ V
$$
be the homomorphism defined by $f\,\longmapsto\, f\cdot s_0$. Consider the torsion part
$Tor\, \subset\, V/\gamma({\mathcal O}_C)$ and consider the inverse image $q^{-1}(Tor)\,\subset\, V$,
where $q\, :\, V\, \longrightarrow\,V/\gamma({\mathcal O}_C)$ is the quotient map; note that
$q^{-1}(Tor)$ is a line subbundle of $V$. Since the homomorphism of fibers
$\gamma_y\, :\, k\, \longrightarrow\, V_y$ over $y$ vanishes, we have $Tor\,\not=\, 0$. Hence
$\text{degree}(q^{-1}(Tor)) \,>\, \text{degree}({\mathcal O}_C)\,=\, 0$. So
$q^{-1}(Tor)\,\subset\,V$ contradicts the fact that $V$ is semistable of degree zero.
Thus \eqref{e1d} is proved.

Note that \eqref{e1d} contradicts \eqref{e1} because $d\, >\,1$. This completes the proof.
\end{proof}

\begin{corollary}\label{cor1}
Nontrivial \'etale covering maps $\widehat{\pi}\,:\, X \,\longrightarrow\, Y$ between irreducible smooth
projective varieties do not admit any \rul vector bundle.
\end{corollary}

\begin{proof}
Take $D\, \subset\, Y$ to be a smooth curve such that the homomorphism of fundamental groups
$\pi_1(D)\, \longrightarrow\, \pi_1(Y)$ is surjective; for example, take $D$ to be a smooth
complete intersection curve obtained by intersecting very ample smooth divisors of $Y$. Set
$C\,=\, \widehat{\pi}^{-1}(D)$ and $\pi\,=\, \widehat{\pi}\big\vert_C$. If $E$ is a \rul 
vector bundle on $X$ for $\widehat{\pi}$, then $E\big\vert_C$ is a \rul vector bundle on
$C$ for ${\pi}$. Now Lemma \ref{nonexistenceetale} completes the proof.
\end{proof}

In the following example, we construct a ramified double covering with no \rul line bundle.

\begin{example} \label{nonexsmall}
Let $\pi\,:\,C \,\longrightarrow\, D$ be a degree $2$ ramified covering
map between smooth projective curves with $g(D)\,\ge\, 1$. Then we have
$\pi_*\cO_C \,\cong\, \cO_D \oplus L^{-1}$ for some line bundle $L$. The branch locus $B$ for $\pi$
satisfies the condition $B \,\in\, |L^{\otimes 2}|$. Assume that $\text{degree}(L)\,=\,1$. Then 
the degree of the ramification divisor for $\pi$ is $2$. By Riemann--Hurwitz formula,
we have 
\begin{equation} \label{RHe100}
	2g(C)-2\,=\,2(2g(D)-2)+2.
\end{equation}
Thus, we have $g(C)\,=\, 2\cdot g(D).$ Let $M \,\in\, \text{Pic}(C)$ be such that
\begin{equation}\label{e3}
\pi_*M\ \cong\ \cO_D^{\oplus 2}.
\end{equation}
{}From \eqref{e3} it follows that
\begin{equation}\label{e2}
\text{degree}(M)\,=\,\chi(M) + g(C)-1\,=\, \chi(\cO_D^{\oplus 
2})+ g(C)-1\,=\, 2(1-g(D))+ g(C)-1\,=\,1.
\end{equation}
Since $g(C) \,\ge\, g(D) \,\ge\, 1$, from \eqref{e2} it follows that $\dim H^0(C,\, M)\, \leq\, 1$.
On the other hand, from \eqref{e3} it follows that $\dim H^0(C,\, M)\,=\, 2$. Consequently,
there is no \rul line bundle $M$ on $C$.
\end{example}

We generalize the Example \ref{nonexsmall} for arbitrary rank vector bundle in the following.

\begin{example}
Consider the morphism $\pi\,:\,C \,\longrightarrow\, D$ from Example \ref{nonexsmall}.
Let $E$ be a vector bundle on $C$ of rank $r \,\ge\, 2$ such that
	\begin{equation}\label{e3m}
		\pi_*E\ \cong\ \cO_D^{\oplus 2r}.
	\end{equation}
	{}It follows from \eqref{e3m}, together with the Riemann-Hurwitz formula in \eqref{RHe100}, that
	\begin{equation}\label{e4m}
		\text{degree}(E)\,=\,\chi(E) + r(g(C)-1)\,=\, \chi(\cO_D^{\oplus 
			2r})+ r(g(C)-1)\,=\, 2r(1-g(D))+ r(g(C)-1)\,=\,r.
	\end{equation}
	From \eqref{e3m}, it also follows that $E$ is generated by its global
sections. Since $r \,\ge\, 2$, a general section of $E$ has no zeros. Thus, we obtain a \ses
\[
0\ \longrightarrow\ \cO_C\ \longrightarrow\ E\ \longrightarrow\ E_1\ \longrightarrow\ 0,
\]
where the cokernel $E_1$ is a vector bundle of rank $(r-1)$. Since $E_1$ is a quotient of $E$, it is
also generated by its global sections. Moreover, $\text{degree}(E_1)\,=\, r$ and 
	\begin{equation}
		\dim H^0(C,\, E_1)\ \ge\ (2r-1).
	\end{equation}
	If $(r-1)\, \ge\, 2$, the above procedure can be repeated. Inductively, we obtain a sequence of vector bundles $E_1,\, E_2,\, \cdots,\, E_{r-1}$ such that $E_{r-1}$ is a line bundle of
degree $r$ and $$h^0(C,\, E_{r-1}) \,\ge\, (r+1).$$
	Since $g(C) \,\ge\, g(D) \,\ge\, 1$, this contradicts the following result.

\textbf{\underline{Claim}:}\, Let $L$ be a degree $r$ line bundle on a smooth, projective curve
$C$ with genus  $ g \,\ge\, 1$. Then $\dim H^0(C,\, L)\, \le\,  r$.
	\begin{proof}[Proof of the claim]
		
\textbf{\underline{Case 1}:}\, Assume that $r \,>\, 2g -2$. Then by Serre duality, we have $H^1(C,\, L)
\,\cong\, H^0(C,\, L^{-1} \otimes K_C) \,=\,0$. Thus, by Riemann-Roch we have
$$h^0(C,\, L)\ =\ deg(L)+(1-g)\ =\ r+1-g\ \le\ r.$$
		
		\textbf{\underline{Case 2}:}\, Assume that $0 \,\le\, r \,\le\, 2g-2$.  In this case, if
$H^1(C,\, L)\,=\,0$, then by the argument in Case 1, we have $\dim H^0(C,\, L) \,\le\, r$.

Assume that $H^1(C,\, L) \,\ne\, 0$. Then by Clifford's theorem (see \cite[p.~343, Theorem 5.4]{Har77}), we have
		\begin{equation} \label{clifford}
\dim H^0(C,\, L)\ \le\ \frac{r}{2}+1.
		\end{equation}
Since $r \,\ge\, 2$, we have $r+2 \,\le\, 2r$, so $1 + \frac{r}{2} \,\le\, r$. Thus, by \eqref{clifford}, we get
that $\dim H^0(C,\, L) \,\le\, r$. This completes the proof of the claim.
\end{proof} 
Therefore, there is no \rul vector bundle on $C$ \wrt $\pi$.
\end{example}

The following example provides a construction of a ramified cyclic Galois covering map
admitting a \rul line bundle.

\begin{example}
This example is inspired by \cite[Section 5]{parameswaran2024ulrich}. Let $f$ be a
homogeneous polynomial in $k[X,\, Y,\, Z]$ of degree $d \cdot\delta \, \in\, 2{\mathbb Z}$.
By \cite[Theorem 5.1]{carlini2008complete}, $f$ can be expressed as $f\,=\,F_1G_1+F_2G_2$, where
$F_1,\, F_2,\, G_1,\, G_2$ are homogeneous with $F_1$ being of degree $\delta$ while $F_2$ and
$G_2$ are of degree $\frac{d\cdot\delta}{2}$. Further, we assume that $F_1$
defines a smooth curve $Z(F_1)\,=\, D$ in $\bP^2$ that intersects both
$Z(F_2)$ and $Z(G_2)$ transversally. This implies that $D \bigcap Z(F_2 \cdot G_2)$ consists of
distinct $d \cdot\delta^2$ points of $D$. Consider the degree $d$ cyclic covering $\pi\,:\, C
\,\longrightarrow\, D$ branched over these $d \cdot\delta^2$ points in $|{\cO_{\bP^2}(d\cdot\delta)}\big\vert_D|$.
Restrict $F_2, G_2$ to $F'_2, G'_2$ in $H^0(D,\, {\cO_{\bP^2}(\frac{d \cdot\delta}{2})}\big\vert_D)$, and define the map 
 	\[\begin{tikzcd}
 		D\, && \, {\mathbb{P}^1}
 		\arrow["{|F_2', G_2'|}", from=1-1, to=1-3]
 	\end{tikzcd}\]
for which the inverse image of $\{0,\, \infty\}$ is precisely the distinct $d \cdot\delta^2$ branch points.
Let $\pi'\,:\, \bP^1 \,\longrightarrow\, \bP^1$ be the degree $d$ cyclic covering map
$(x,\, y)\, \longmapsto\, (x^d,\, y^d)$ that is branched over $\{0,\, \infty\}$.
Consider the following Cartesian diagram:
 	\[\begin{tikzcd}
 		{D \times_{\mathbb{P}^1} \mathbb{P}^1} && {\mathbb{P}^1} \\
 		D && {\mathbb{P}^1}
 		\arrow["j", from=1-1, to=1-3]
 		\arrow["{\pi''}"', from=1-1, to=2-1]
 		\arrow["{\pi'}"', from=1-3, to=2-3]
 		\arrow["{|F_2', G_2'|}", from=2-1, to=2-3]
 	\end{tikzcd}\]
{}From the uniqueness of cyclic coverings with a fixed branch divisor
it follows that ${D \times_{\mathbb{P}^1} \mathbb{P}^1} \,\cong\, C$ with $\pi''\,=\,\pi$.
Since $\cO_{\bP^1}(d-1)$ is a \rul line bundle for $\pi'$, by base change theorem we
conclude that $j^*\cO_{\bP^1}(d-1)$ on $C$ is a \rul line bundle \wrt $\pi$.
 \end{example}

Recall that a Noetherian ring $R$ is said to satisfy condition $S_2$ if, for every prime ideal $\mathfrak p
\,\subset\, R$, we have $\text{depth}\, (R_{\mathfrak p}) \,\ge\, \min \{2,\,\dim R_{\mathfrak p}\}$. A variety
$Y$ is said to satisfy the condition $S_2$ if every stalk $\cO_{Y, y}$ of it satisfies $S_2$. By Serre's criterion, a variety is normal if and only if it satisfies properties $R_1$ and $S_2$,  where $R_1$ means that the local ring of the variety at every codimension $1$ point is regular.

\begin{proposition} \label{mostcrucialprop1}
 	Let $\pi\,:\,X \longrightarrow Y$ be a finite morphism of varieties. Let $E$ be a coherent sheaf on
$X$ such that $\pi_*E$ is a locally free sheaf of positive rank on $Y$. Then the following
three statements hold:
 	\begin{enumerate}
 		\item [(1)] The morphism $\pi$ is surjective, and $E$ is torsionfree. 
 		\item [(2)] If $X$ is normal and $Y$ satisfies condition $S_2$, then $E$ is a reflexive sheaf. 
 		\item [(3)] If  $X$ is smooth and $Y$ is Cohen-Macaulay, then $E$ is a vector bundle.
 	\end{enumerate}
 \end{proposition}
 
 \begin{proof}
 	Since $\pi_* E$ is a locally free sheaf on $Y$, it follows that $supp(\pi_*E)$ is $Y$. Using the given condition
    that $\pi$ is a finite map it is deduced that $\pi(supp(E))\,=\,supp(\pi_*E)\,=\,Y$. Consequently, we have $\pi(X)\,=\,Y$, and thus $\pi$ is surjective. 
 	
 	Let $T \,\subset\, E$ be the torsion subsheaf of $E$. If $T \ne 0$, then $supp(T)$ is a proper closed subscheme of $X$. Since $\pi$ is finite, $\pi$ is a proper map. Moreover, since $X$ and $Y$ are integral, and $\pi$ is surjective, the image  $\pi(supp(T))$ is a proper closed subscheme of $Y$. Thus, $\pi_*T$ is a torsion sheaf on $Y$. Since $T \longhookrightarrow E$ is injective, we have $\pi_*T \longhookrightarrow \pi_*E$ is injective. Since $\pi_* E$ is locally free, and hence the torsion sheaf $\pi_*T=0$. Since $\pi$ is finite, the functor $\pi_*$ is faithful on coherent sheaves. Therefore $T=0$. Thus $E$ is torsionfree. This proves statement (1).
 	
 	The proofs of statement (2) and statement (3) follow arguments similar to those in \cite[Proposition 4.6]{parameswaran2024ulrich}. Let $x\,=\,x_0$ be a point in $X$ and $\pi(x)\,=\,y$ be its image in $Y$. Let $\widehat{\cO_{Y, y}} \,\,\longrightarrow\,\, \widehat{\cO_{X, x_i}}$ be the finite map of complete local rings induced by the structure map, where $x_i \,\in\, \pi^{-1} \{y\}$. Since $\pi$ is finite, after completing at $y$, we get the following isomorphism of $\widehat{\cO_{Y, y}}$ modules
 	$$\widehat{(\pi_*E)_y}\,\cong\, \oplus_{x_i \in \pi^{-1} \{y\}} \widehat{E_{x_i}}.$$
 	As $(\pi_*E)_y$ is a free $\cO_{Y, y}$ module, the completion $\widehat{(\pi_*E)_y}$ is a free $\widehat{\cO_{Y, y}}$ module. Given that $\widehat{E_{x_i}}$ are direct summands of a free module, it is deduced that $\widehat{E_{x_i}}$ is a free $\widehat{\cO_{Y, y}}$ module for each $x_i \,\in\, \pi^{-1} \{y\}$. Thus, the projective dimension of $\widehat{E_{x_i}}$  as an $\widehat{\cO_{Y, y}}$ module is zero. Thus, by the Auslander-Bucshbaum theorem, we have 
 	\begin{equation} \label{firstdee1}
 		\text{depth}_{\widehat{\cO_{Y, y}}} (\widehat{E_{x_i}})\ =\ \text{depth}\, (\widehat{\cO_{Y, y}})\,\,\ge\,\, \min \{2, \,\dim \cO_{Y, y}\},
 	\end{equation}
 	 because $Y$ satisfies $S_2$. Since $\widehat{\cO_{Y, y}} \longrightarrow \widehat{\cO_{X, x_i}}$ is a finite map, the $\text{depth} (\widehat{E_{x_i}})$ is the same whether it is considered as a module over $\widehat{\cO_{Y, y}}$ or over  $\widehat{\cO_{X, x_i}}$; see \cite[\href{https://stacks.math.columbia.edu/tag/0AUK}{Lemma 0AUK}]{stacks-project}. Since $\pi$ is finite and surjective, we have $\dim (\cO_{Y, y})\,=\,\dim (\cO_{X, x_i})$. Therefore, we get
 	\begin{equation} \label{depthe1}
 		\text{depth}_{\widehat{\cO_{X, x_i}}} (\widehat{E_{x_i}})\, \ge \,\min \{2, \dim \cO_{Y, y}\}\,=\,\min \{2, \dim \cO_{X, x_i}\}.
 	\end{equation}
 	Taking $x=x_0$, and using preservation of depth under completion, we obtain $\text{depth}_{\cO_{X, x}} (E_{x})\, \ge \,\min \{2, \dim \cO_{X, x}\}$ for all points $x$ in $X$. Since $E$ is torsionfree by $(1)$, we have $\dim (supp(E_x))\,=\,\dim \cO_{X, x}$. Hence $E$ satisfies property $S_2$ \cite[\href{https://stacks.math.columbia.edu/tag/0340}{Section 0340}]{stacks-project}. 
 	Since $E$ is torsionfree and satisfies $S_2$, and $X$ is normal, we have $E$ is a reflexive sheaf on $X$ \cite[Lemma 31.13.13, \href{https://stacks.math.columbia.edu/tag/0AVT}{Section 0AVT}]{stacks-project}. 
 	This proves $(2)$.
 	
 	For $(3)$, we assume $X$ is smooth and $Y$ is Cohen-Macaulay. Since $Y$ is Cohen-Macaulay, $\text{depth}\, (\widehat{\cO_{Y, y}})\,=\,\dim \cO_{Y, y}$. Thus, by \eqref{firstdee1}, we have $\text{depth}_{\widehat{\cO_{Y, y}}} (\widehat{E_{x_i}})\,=\,\dim \cO_{Y, y}$. The same depth comparison as for $(2)$ then gives
 	\begin{equation} \label{thirddee2}
 		\text{depth}_{\widehat{\cO_{X, x}}} (\widehat{E_{x}})\,=\,\dim \cO_{X, x},
 	\end{equation} 
 	 for all points $x$ in $X$. Since $X$ is smooth, the projective dimension $\text{pd}_{ \cO_{X, x}} (E_{x})$ is finite. Using \eqref{thirddee2} and the Auslander-Buchsbaum formula, it follows that $\text{pd}_{ \cO_{X, x}} (E_{x})=0$ for all $x \in X$. Thus, $E_{x}$ is a free $\cO_{X, x}$ module for all $x \in X$, and hence, $E$ is locally free. This proves $(3)$.
 	\end{proof}

\section{Necessary conditions for relatively Ulrich sheaves} \label{section2}

Let $\pi\,:\,X\, \longrightarrow\, Y$ be a finite surjective morphism of projective varieties. 
Then $\pi$ is called an abelian Galois covering map if a finite abelian group $G$ acts faithfully on $X$
such that $Y\,=\, X/G$ with $\pi$ being the quotient map. The group $G$ is called the Galois group
of $\pi$.

Assume that $\pi$ is an abelian Galois covering map, with Galois group $G$.
Since $G$ is finite abelian, there exists a chain of subgroups
$$\{1\}\ =\ G_0\ \subset\ G_1\ \subset\ \cdots\ \subset\ G_l\ =\ G$$ such that each quotient $G_i/G_{i-1}$ is
cyclic group of order $d_i$. Setting $X_i\,:=\,X/G_i$, we obtain a factorization
\[\begin{tikzcd}
	{X=X_0} & {X_1} & {X_2} & \cdots & {X_{l-1}} & {X_{l}=Y}
	\arrow["{\pi_1}", from=1-1, to=1-2]
	\arrow["{\pi_2}", from=1-2, to=1-3]
	\arrow["{\pi_3}", from=1-3, to=1-4]
	\arrow[from=1-4, to=1-5]
	\arrow["{\pi_l}", from=1-5, to=1-6],
\end{tikzcd}\]	
where, for each $1 \,\le\, i \,\le\, l$, the morphism $\pi_i\,:\,X_{i-1}\, \longrightarrow\, X_i$ is a cyclic
Galois covering map of degree $d_i$. Thus, $\pi$ factors as a tower of finitely many cyclic Galois covers of
degrees $d_1$, $d_2$, $\cdots$, $d_l$; namely, 
$$\pi\ =\ \pi_l \circ \pi_{l-1} \circ \cdots \circ \pi_2 \circ \pi_1.$$
 The cyclic groups $Gi/G_{i-1}$ are identified with $\bZ/d_i \bZ$. For each $1 \,\le\, i
\,\le\, l$, the action of $Gi/G_{i-1}$ on $X_{i-1}$ induces the following decomposition 
$$(\pi_i)_* \cO_{X_{i-1}} \ \cong\  \bigoplus_{\chi \in (\bZ/d_i \bZ)^*} \cF_{\chi},\
\,\,\,\,\cF_1\,=\,\cO_{X_i},$$
where $(\bZ/d_i \bZ)^*$ denotes the character group, $1$ denotes the trivial character, and
$\cF_{\chi}$ are rank $1$ torsionfree sheaves on $X_i$ \cite[\S~1]{alexeevabelian}.

\begin{remark}
By \cite[Lemma 1.1]{alexeevabelian}, the sheaf $\cO_{X_{i-1}}$ is $S_2$ if and only if every $\cF_{\chi}$ is 
$S_2$. Assume that $X_{i-1}$ is $S_2$ and $X_i$ is normal. Since $\cF_{\chi}$ is torsionfree and $S_2$, it 
follows that $\cF_{\chi}$ is reflexive \cite[\href{https://stacks.math.columbia.edu/tag/0AVT}{Section 0AVT}, 
Lemma 31.13.13]{stacks-project}. Since reflexive rank $1$ sheaves on a normal variety are divisorial, we have 
$\cF_{\chi} \cong \cO_{X_i}(D)$ for some Weil divisor $D$ on $X_i$. If, moreover, $X_i$ is smooth, then $D$ 
is a Cartier divisor, and thus, $\cF_{\chi}$ are line bundles on $X_i$. Therefore, in particular, if 
$X_{i-1}$ is normal and $X_i$ is smooth, then $\cF_{\chi}$ are line bundles. Hence $(\pi_i)_* \cO_{X_{i-1}}$ 
is a vector bundle, and consequently $\pi_{i}$ is a flat morphism.
\end{remark}

\begin{definition} [Standard abelian cover] \label{standardabelian}
The morphism $\pi_i\,:\,X_{i-1} \,\longrightarrow\, X_i$ is called a standard cyclic cover of degree
$d_i$ if it  is a ramified cyclic cover and
$$ (\pi_i)_* \cO_{X_{i-1}} \,\cong\, \cO_{X_i} \oplus N_i^{-1} \oplus N_i^{-2} \oplus \cdots \oplus N_i^{-d_i+1}$$
for some line bundle $N_i$ on $X_{i}$. The morphism $\pi\,:\,X \,\longrightarrow\, Y$ is called a standard
abelian Galois cover if the morphism $\pi_i$ is a standard cyclic cover for every $1 \,\le\, i \,\le\, l$.
\end{definition}

\begin{remark} \label{standardflat}
A standard abelian Galois cover may be viewed as a tower of standard cyclic covers.  Note that if 
$\pi_i\,:\,X_{i-1}\, \longrightarrow\, X_i$ is a standard cyclic cover, then $(\pi_i)_* \cO_{X_{i-1}}$ is vector 
bundle. Hence $\pi_i$ is flat. Therefore the composition $\pi\,=\,\pi_l \circ \cdots\circ \pi_1\,:\,X 
\,\longrightarrow\, Y$ is also flat. Thus every standard abelian cover is flat.
\end{remark}

The branch divisor $B_i$ of $\pi_i$ is the zero locus $Z(b_i)$ of a nonzero section $b_i \,\in\, H^0(X_i, 
\, N_i^{\otimes d_i})$. The pair $(b_i,\, N_i)$, where $0 \,\ne\, b_i \,\in\, H^0(X_i,\, N_i^{\otimes d_i})$,
uniquely  determines --- up to isomorphism --- the standard cyclic cover $\pi_i\,:\,X_{i-1}
\,\longrightarrow\, X_i$.

For each $1 \,\le\, i \,\le\, (l-2)$, let $$\pi_{i+1}'\,:=\,( \pi_l \circ \pi_{l-1} \circ \ldots \circ \pi_{i+2} \circ 
\pi_{i+1} )\,:\, X_i \,\longrightarrow\, Y$$ denote the composition of the morphisms. We have
the diagram:
\footnotesize
\begin{equation} \label{keydiag10}
	\begin{tikzcd}
		X && {X_{i-1}} && {X_i} && Y
		\arrow[from=1-1, to=1-3]
		\arrow["\pi", bend left=35, from=1-1, to=1-7]
		\arrow["{\pi_i}", from=1-3, to=1-5]
		\arrow["{\pi_i'}"', bend right=35, from=1-3, to=1-7]
		\arrow["{\pi_{i+1}'}", from=1-5, to=1-7]
	\end{tikzcd}
\end{equation}
\normalsize

\begin{definition}[{Admissible abelian Galois cover}]\label{admissibleab}
An admissible abelian Galois cover $\pi\,:\,X \,\longrightarrow\, Y$ is a standard abelian Galois
cover such that, for each $1 \,\le\, i \,\le\, l$, the line bundle $N_i$ on $X_i$ descends to $Y$; in
other words, there exists a line bundle $M_i$ on $Y$ such that $N_i \,\cong\, (\pi_{i+1}')^* M_i$.
Moreover, the branch section $b_i \,\in\, H^0(X_i,\, N_i^{\otimes d_i})$ is the pullback of a section
on $Y$, meaning $b_i\,=\, (\pi_{i+1}')^*s_i$ for some $s_i \,\in\, H^0(Y,\,M_i^{\otimes d_i})$.
\end{definition}

\begin{remark}
	It follows that every admissible abelian Galois cover $\pi:X \longrightarrow Y$ is flat by Remark \ref{standardflat}. Also note that every standard cyclic Galois cover is an admissible abelian Galois cover.
\end{remark}

We denote the branch divisors $Z(s_i)$ on $Y$ by $B_i'$.  
\begin{example}
	We consider the tuple
	\begin{equation}\label{a2}
		(B_1',\, B_2',\, \cdots,\, B_l') \ \in\ |M_1^{\otimes d_1}| \times |M_2^{\otimes d_2}| \times \cdots
		\times |M_l^{\otimes d_l}|,
	\end{equation}
	where $M_1,\, M_2,\, \cdots,\, M_l$ are line bundles on $Y$. Then an admissible abelian Galois cover $\pi$ can be constructed by successively composing the standard cyclic covers determined by the pairs $(B_i\,=\,\pi_{i+1}'^*B_i', N_i\,=\,\pi_{i+1}'^* M_i)$.
\end{example}

For an admissible abelian Galois covering map $\pi\,:\,X \,\longrightarrow\, Y$ between projective varieties, 
we prove that the existence of \rul vector bundles on $X$ \wrt $\pi$ implies that the line bundles $M_i$ (see 
Definition \ref{admissibleab}) are globally generated. Further, we prove that their existence implies that 
the branch divisor $s_i \,\in\, H^0(Y,\, M_i^{\otimes d_i})$ is in the image of the $d_i$-fold multiplication 
map.

Let
\begin{equation}\label{a1}
\pi\ :\ X \ \longrightarrow\ Y
\end{equation}
be an admissible abelian Galois covering map of degree $d$ between 
projective varieties of dimension $n$. 

The sheaf
$\pi_* {\cO_X}$ splits as direct sum of line bundles
\begin{equation}\label{a3}
\pi_*{\cO_X}\, \cong\, \cO_Y \oplus \bigg(M_1^{-1} \oplus M_1^{-2} \oplus \cdots \oplus M_1^{-d_1+1}\oplus
M_2^{-1} \oplus M_2^{-2} \oplus \cdots \oplus M_2^{-d_2+1}
\end{equation}
$$
\oplus \cdots \cdots \oplus M_l^{-1} \oplus M_l^{-2} \oplus \cdots \oplus M_l^{-d_l+1}\bigg)
\oplus \bigg(\sum_{1 \le i <j }^{j \le l}\sum_{k_i, k_j=(-1, -1)}^{k_i, k_j=(-d_i+1, -d_j+1)}
M_i^{k_i} \otimes M_j^{\otimes k_j} \bigg)
$$
$$
 \oplus \cdots \oplus (M_1^{-d_1+1} \otimes M_2^{-d_2+1} \otimes \cdots \otimes M_l^{-d_l+1}).$$

For the remainder of this section $X$, $Y$ and $X_i$ are assumed to be projective varieties, unless otherwise stated.

\begin{lemma} \label{torsionfreeinjective}
	Let $L$ be a line bundle on $X$, and let $\cO_X\, \longhookrightarrow\, L$ be the morphism defined by a nonzero section of $L$. For any torsionfree sheaf  $E$ on $X$, tensoring with $E$ induces an injective morphism $E \longhookrightarrow E \otimes L$.
\end{lemma}

\begin{proof}
We choose an affine open cover $\{U_i\,=\,\spec (A_i)\}$ of $X$ that trivializes the line bundle $L$. Let $s
\,\ne\, 0$ be the given nonzero section. Under a trivialization $L \vert_{U_i} \,\cong\, \cO_{U_i}$, the
section $s$ corresponds to an element $f_i \,\in\, A_i$. Thus, locally on $U_i$, the morphism $\cO_X
\,\longhookrightarrow\, L$ is given by $1\, \longrightarrow\, f_i$. Since $X$ is integral and $s\,\ne\, 0$, the
element $f_i$ is nonzero in $A_i$. On $U_i$, the coherent sheaf $E$ is represented by an $A_i$ module
$P_i$. Under the chosen trivialization of $L\big\vert_{U_i}$, the morphism $E \,\longhookrightarrow\,
E \otimes L$ is given locally by $$P_i \ \longrightarrow\ P_i \otimes_{A_i} A_i\ \cong\ P_i, \ \ \,
a\ \longmapsto\ f_i \cdot a.$$
Since $E$ is torsionfree, $P_i$ is a torsionfree $A_i$ module. Thus, as $f_i \,\ne\, 0$, the equality
$f_i \cdot a\,=\,0$ implies that $a\,=\,0$. Thus, the morphism $E \,\longrightarrow \,E \otimes L$ is
injective locally on each $U_i$. Therefore $E \,\longrightarrow\, E \otimes L$ is injective.
\end{proof}

\begin{proposition} \label{effectivelinebundle}
Let $E$ be a \rul coherent sheaf on the admissible abelian Galois  cover $X$ \wrt $\pi$ in Definition \ref{admissibleab}. Then $H^0(Y,\, M_i)
\,\ne\, 0$ for all $i\,=\,1,\, 2,\, \cdots,\, l$.
\end{proposition}

\begin{proof}
Since $\pi\,:\,X \,\longrightarrow\, Y$ is an admissible abelian Galois, the map $\pi_i\,:\,X_{i-1}
\,\longrightarrow\, X_i$ is a standard cyclic cover for each $1\, \le\, i \,\le\, l$. Therefore, by 
construction, the tautological section defines a nonzero section $0 \,\ne\, t_i \,\in\, H^0(X_{i-1},\,
\pi_i'^*M_i)$ whose  $d_i$-{th} power recovers pullback of the branch section $b_i\,=\,\pi_{i+1}'^* s_i$
defining $B_i\,=\,\pi_{i+1}'^*B_i'$.
Thus, we get that $$H^0(X,\, \pi^*M_i) \ \ne\ 0$$
for all $1\, \leq\, i\, \leq\, l$ (see the diagram \eqref{keydiag10}). A nonzero section of $\pi^*M_i$
defines an injective map 
\begin{equation} \label{injectivee1}
	\cO_X\ \longrightarrow\ \pi^*M_i .
\end{equation}
By Proposition \ref{mostcrucialprop1}, the relatively Ulrich sheaf $E$ is a torsionfree sheaf on $X$. Thus, by Lemma \ref{torsionfreeinjective}, tensoring \eqref{injectivee1} with $E$ on $X$ gives an injective morphism
$$E\ \longhookrightarrow\ E \otimes \pi^*M_i .$$
By applying $\pi_*$, and using the projection formula, we get an injective map
$$\cO_Y^{\oplus d \cdot \text{rank}(E)}\ \longhookrightarrow\ M_i \otimes (\cO_Y^{\oplus d \cdot \text{rank}(E)} ),$$
where $d\,=\, \text{degree}(\pi)$. This implies that
$$
H^0(Y,\, \cO_Y)^{\oplus d \cdot \text{rank}(E)}\,=\,
H^0(Y,\, \cO_Y^{\oplus d \cdot \text{rank}(E)}) \, \longhookrightarrow\, H^0(Y,\,
M_i \otimes (\cO_Y^{\oplus d \cdot \text{rank}(E)}))\,=\, H^0(Y,\, M_i)^{\oplus d \cdot \text{rank}(E)},
$$
which in turn shows that $H^0(Y,\, M_i)\, \not=\, 0$.
\end{proof}

\begin{remark}
Proposition \ref{effectivelinebundle} may not hold in general for arbitrary abelian covers. One may also construct an abelian covering by considering the composition of two degree $2$ cyclic coverings 
$\pi_1\,:\, X_1\,\longrightarrow\, Y$ and $\pi_2\,:\, X \,\longrightarrow\, X_1$ such that
$(\pi_1^*B_1) \cap B_2\,=\,\emptyset$, where $B_1 \,\in\, |L_1^{\otimes 2}|$ is the branch divisor of $\pi_1$
for a line bundle $L_1$ on $Y$ while $B_2\,\in\, |L_2^{\otimes 2}|$ is the branch divisor of $\pi_2$ for a
line bundle $L_2$ on $X_1$. We assume that $L_2$ is not the pullback of some line bundle on $Y$. Then the resulting cover $\pi\,=\,\pi_1\circ \pi_2\,:\, X \,\longrightarrow\,
Y$ is a standard abelian Galois cover but not an admissible abelian cover. If there
exists a \rul coherent sheaf on $X$ \wrt the covering map $\pi\,:\, X \,\longrightarrow\,
Y$, then by the same argument as in Proposition \ref{effectivelinebundle}, we obtain $H^0(Y,\, L_1) \,\ne \,0$. However, in general, the same argument does not imply that $H^0(X_1,\, L_2) \,\ne \,0$ .
\end{remark}

For the following proposition, we are assuming that $\pi\,:\,X \,\longrightarrow\, Y$ is
a standard cyclic Galois covering map of degree $d:\,=\, d_1$ between projective varieties. Thus
\begin{equation}\label{a4}
\pi_*{\cO_X}\ \cong\ \cO_Y \oplus M^{-1} \oplus M^{-2} \oplus \cdots \oplus M^{1-d},
\end{equation}
where $M:\, =\, M_1$ is a line bundle on $Y$ (see Definition \ref{standardabelian}). Thus $\pi$ is flat.
The branch divisor is given by $B \,\in\, |M^{\otimes d}|$. 

\begin{proposition}
\label{ggcyclic}
Let $E$ be a \rul vector bundle on $X$ \wrt $\pi$. We assume that the divisor $B$ is reduced. Then the line bundle $M$ (see \eqref{a4})
is generated by its global sections.
\end{proposition}

\begin{proof}
Let $P \,\in\, Y$ be a closed point. Consider the following Cartesian diagram
\[\begin{tikzcd}
Z && X \\
\spec(k) && Y
\arrow[hook, from=1-1, to=1-3]
\arrow["{\pi'}"', from=1-1, to=2-1]
\arrow["\pi", from=1-3, to=2-3]
\arrow["P", hook, from=2-1, to=2-3]
\end{tikzcd}\]
where $Z$ is a scheme of distinct $d$ points if $P \,\not\in\, B$, and $Z \,\cong\,
\text{spec}(\frac{K[x]}{x^d})$ if $P \,\in\, B$. Consider the restriction map 
\begin{equation}\label{a6}
\pi^*M\ \longrightarrow\ {\pi^*M}\big\vert_Z .
\end{equation}
It will be shown that the induced map of cohomologies
\begin{equation}\label{a7}
H^0(X,\, \pi^*M) \, \longrightarrow\, H^0(X,\, {\pi^*M}\big\vert_Z)\,
=\, H^0(Z,\, {\pi^*M}\big\vert_Z)
\end{equation}
is nonzero.

To see this, from the construction of $X$ it follows that there exist $t \,\in\, H^0(X,\, \pi^*M)$ such that
\begin{equation}\label{a5}
\text{Div}(t^d)\ =\ \pi^*B\ =\, d \cdot (\pi^* B)_{\text{red}},
\end{equation}
since $B$ is reduced, where $(\pi^* B)_{\text{red}}$ is the reduced induced scheme structure on $\pi^*B$. Also, with a mild abuse of 
notation, the preimage of a point $Q \,\in\, B$ in $(\pi^* B)_{\text{red}}$ will be denoted by $Q$. So $\pi^*Q 
\, = \, d \cdot Q$.

The section $t$ restricts to a nonzero element in
$H^0(Z,\, {\pi^*M}\big\vert_Z)$, because if $P\,=\,P_i$, and $t$ maps to zero,
then $dP_i$ is in the support of $t$, and if $P \,\not\in\, B$, then $d$ distinct point outside the
ramification divisor will be in the support of $t$; note that neither of them is possible, since the divisor for
$t^d$ is $d \cdot (\pi^* B)_{\text{red}}$ (see \eqref{a5}). This proves the above statement
that the homomorphism in \eqref{a7} is nonzero.

Now tensoring \eqref{a6}
with the \rul vector bundle $E$, we get the restriction map
\begin{equation}\label{a8}
E \otimes{\pi^*M}\ \longrightarrow\ E \otimes {\pi^*M}\big\vert_Z.
\end{equation}
Since $E$ is locally free, $E \vert_{Z}$ is locally free over $\cO_Z$. The section $t \vert_{Z} \in H^0(\cO_Z, {\pi^*M}\big\vert_Z)$ is nonzero, hence multiplication by $t \vert_{Z}$ on $E \vert_{Z}$ is not the zero map. Since $E$ is globally generated,  some global sections of $E$ restricts to a section $E \vert_{Z}$ not annihilated by $t \vert_{Z}$. Thus, the induced map of global sections
\begin{equation} \label{a10}
H^0(X,\, E \otimes {\pi^*M})\ \longrightarrow\ H^0(Z,\, E \otimes {\pi^*M}\big\vert_Z)
\end{equation}
remains nonzero. 

Applying $\pi_*$ to \eqref{a8}, and using the projection formula, we have
\begin{equation}\label{a9}
M \otimes \cO_Y^{\oplus d \cdot \text{rank}(E)}\ \longrightarrow\ M_{|P}^{\oplus d \cdot \text{rank}(E)}.
\end{equation}
The induced map on global sections
$$
H^0(Y,\, M \otimes \cO_Y^{\oplus d \cdot \text{rank}(E)})\ \longrightarrow\ M_{|P}^{\oplus d \cdot \text{rank}(E)}
$$
for the homomorphism in \eqref{a9} remains nonzero because
the homomorphism in \eqref{a10} is nonzero. Consequently, the evaluation homomorphism
$H^0(Y,\,M)\, \longrightarrow M_P$ is nonzero. Since $\dim M_P\,=\,1$, this implies that the
evaluation homomorphism $H^0(Y,\,M)\, \longrightarrow M_P$ is surjective.
\end{proof}

\begin{remark}
In Proposition \ref{ggcyclic}, the reducedness of the branch divisor $B$ was used in deriving \eqref{a5}. 
This equation may fail if $B$ is not reduced. It was then used to show that the tautological section $t \in 
H^0(X,\, \pi^*M)$ restricts to a nonzero section in $ 0 \,\ne\, t \vert_{Z} \,\in\, H^0(\cO_Z,\, \pi^*M \vert_{Z})$.
	
However, one can still prove that the section $t$ restricts to a nonzero section in $ 0 \ne t \vert_{Z} \in 
H^0(\cO_Z,\, \pi^*M \vert_{Z})$ without assuming that $B$ is reduced. The alternative argument avoids the above use 
of \eqref{a5} to keep track of cycle theoretic multiplicities. We give this argument in Proposition 
\ref{admissiblegg}, where the divisors $B_i'$ on $Y$ are not assumed to be reduced.
\end{remark}

\begin{proposition} \label{admissiblegg}
Let $\pi\,:\,X \,\longrightarrow\, Y$ be an admissible abelian covering map of projective varieties. Let $E$ be
a \rul bundle on $X$. Then the line bundles $M_i$ are globally generated for all $i$.	
\end{proposition}

\begin{proof}
	The map $\pi\,:\,X \,\longrightarrow\, Y$ is a composition of standard cyclic coverings
	\[\begin{tikzcd}
		{X=X_0} & {X_1} & {X_2} & \ldots & {X_{l-1}} & {X_{l}=Y}
		\arrow["{\pi_1}", from=1-1, to=1-2]
		\arrow["{\pi_2}", from=1-2, to=1-3]
		\arrow["{\pi_3}", from=1-3, to=1-4]
		\arrow[from=1-4, to=1-5]
		\arrow["{\pi_l}", from=1-5, to=1-6]
	\end{tikzcd}\]	
	where $\pi_i\,:\,X_{i-1} \,\longrightarrow\, X_{i}$ is a standard cyclic covering map of degree $d_i$. For
each $1 \,\le\, i \,\le\, (l-1)$, let $\pi_{i+1}'$ be the map $( \pi_l  \circ \pi_{l-1} \circ
\ldots \circ \pi_{i+2} \circ \pi_{i+1} )\,: \,X_i\, \longrightarrow\, Y$. Thus, $\pi'_{i}$ is the map
$\pi_{i+1}' \circ \pi_i\,:\, X_{i-1} \,\longrightarrow\, Y$. The cyclic covering map
$\pi_i\,:\,X_{i-1} \,\longrightarrow\, X_i$ of degree $d_i$ is branched over $\pi_{i+1}'^* B_i'$ given by
a section $b_i$ in $H^0(X_i,\,\pi_{i+1}'^* M_i^{\otimes d_i})$. We take a point $P \,\in\, Y$, and let $Z''$ be
the scheme-theoretic fiber over $P$ for $\pi_i'\, : \, X_{i-1} \,\longrightarrow\, Y$.  
	\footnotesize
	\[\begin{tikzcd}
		{Z''} && {X_{i-1}} \\
		{Z'} && {X_i} \\
		\spec(k) && Y
		\arrow[hook, from=1-1, to=1-3]
		\arrow[from=1-1, to=2-1]
		\arrow["{\pi_i}"', from=1-3, to=2-3]
		\arrow[hook, from=2-1, to=2-3]
		\arrow[from=2-1, to=3-1]
		\arrow["{\pi_{i+1}'}"', from=2-3, to=3-3]
		\arrow["P", from=3-1, to=3-3]
		\arrow["\pi_{i}'", bend left=45, from=1-3, to=3-3]
	\end{tikzcd}\]
	\normalsize
	We consider the following restriction map
	\begin{equation}\label{gggeneral1}
		\pi_i'^* M_i\ =\ \pi_i^*(\pi_{i+1}'^*M_i) \ \longrightarrow\
{{\pi'_i}^* M_i}\big\vert_{Z''}\ =\ {\pi_i^*({\pi_{i+1}'}^*M_i)}\big\vert_{Z''}.
	\end{equation}
Let $L$ be the line bundle $\pi_{i+1}'^*M_i$. Locally on an affine open $U\, =\, \spec A \subset X_i$ on which 
$L$ is trivial, the standard cyclic cover $\pi_i$ has the form $\pi_i^{-1}(U)\,\cong\, \spec A[T]/(T^{d_i}-b_i 
\vert_{U})$, where $b_i \vert_{U} \,\in\, A$ is the restriction of the section $b_i$ defining $\pi_{i+1}'^* 
B_i'$. The tautological section $t$ is represented by $T$.
	
After restricting to $Z''$, we get $\mathcal O_{Z''\cap \pi_i^{-1}(U)}\, \cong\, R[T]/(T^{d_i}-\overline {b_i 
\vert_{U}})$, where $R\,:=\,\Gamma(U \cap Z',\,\cO_{Z'})$ is the coordinate ring of $U \cap Z'$ and $\overline 
{b_i \vert_{U}}$ is the restriction of $b_i \vert_{U}$ to $Z'$. The module $\mathcal O_{Z''\cap 
\pi_i^{-1}(U)}$ is a free $R$-module with basis $1,\,T,\,\cdots,\,T^{d_i-1}$. Hence, the restriction
$t\big\vert_{Z''}$ 
which is represented by $T$ is nonzero on every affine $Z''\cap \pi_i^{-1}(U)$. Thus, $t \,\in\, H^0(X_{i-1}, 
\, \pi_i'^* M_i)$ restricts to $ 0 \,\ne\, t\big\vert_{Z''}\, \in\, H^0(Z'',\, \pi_i'^* M_i \vert_{Z''})$.

Let $\sigma$ denote the composition $(\pi_{i-1} \circ \pi_{i-2} \circ \ldots \circ \pi_1)\,:\,X 
\,\longrightarrow\, X_{i-1}$. Since each $\pi_i$ is flat, the morphism $\sigma$ is flat. Hence $\sigma_*E$ is 
locally free on $X_{i-1}$. Since $E$ is a \rul bundle on $X$ \wrt $\pi$, the bundle $E':=\sigma_* E$ is \rul 
\wrt $\pi_i'$. Thus, tensoring \eqref{gggeneral1} with $E'$ and applying $(\pi_i')_*$, we get $H^0(M_i) 
\, \longrightarrow\, (M_i)_P$ is nonzero (for more details refer to the last paragraph of the proof of Proposition 
\ref{ggcyclic}).
\end{proof}

In the following proposition, we prove the crucial necessary criterion for the existence of \rul sheaves. 

\begin{proposition} \label{dfoldsurj}
Let $\pi\,:\, X \,\longrightarrow\, Y$ be an admissible abelian covering map of projective varieties. Let $E$
be a \rul coherent sheaf on $X$ \wrt
$\pi$. Then for all $1 \,\le\, i \,\le\, l$, the section $s_i \,\in\, H^0(Y,\, M_i^{\otimes d_i})$ defining
the branch divisor $B_i'$ on $Y$ lies in the image of the natural $d_i$-fold multiplication map
	$$H^0(Y,\, M_i)^{\otimes d_i}\ \longrightarrow\ H^0(Y,\, M_i^{\otimes d_i}).$$
\end{proposition}

\begin{proof}
Let $E$ be a \rul coherent sheaf on $X$ \wrt $\pi$. Thus we have $\pi_* E \,\cong\, \cO_Y^{\oplus m}$ for 
some $m \,\in\, \bN$. It follows that $\End (\pi_*E)$ admits a $\pi_* \cO_X$ algebra structure. From the 
decomposition of $\pi_* \cO_X$ (see \eqref{a3}), this is equivalent to $\cO_Y$--module homomorphisms $$\phi_i\ :\ 
\cO_Y^{\oplus m} \ \longrightarrow\ \cO_Y^{\oplus m} \otimes M_i$$ that satisfies the characteristic 
equations $\phi_i^{d_i}\,=\,s_i \cdot \id_{\cO_Y^{\oplus m}}$ as homomorphisms $\cO_Y^{\oplus m}\,
\longrightarrow\, \cO_Y^{\oplus m} \otimes (M_i^{\otimes d_i})$
for every $1 \,\le\, i \,\le\, l$. Now, the morphism $\phi_i$ is given by a matrix 
$A_i$ of size $m$ whose entries are elements of $H^0(Y,\, M_i)$. The characteristic equation becomes 
$$A_i^{d_i} \ =\ s_i \cdot \id_{m \times m}.$$ This implies that $s_i$ can be written as $s_i\, =\, \sum_{k} a^k_1 \otimes a^k_2 
\otimes \cdots \otimes a^k_{d_i}$, where $a^k_j \,\in\, H^0(Y,\, M_i)$ for $1 \,\le\,j\,\le\,d_i$ and for all $k$.
\end{proof}

\section{Sufficient criterion for \rul sheaves}

In this section, for a standard abelian cover $\pi$, we will prove that the property of the branch divisors 
$s_i \,\in\, H^0(Y, \, M_i^{ \otimes d_i})$ is in the image of the $d_i$-fold multiplication map is sufficient 
for the existence of \rul sheaves. This criterion for the existence of \rul sheaves is related to the normal generation of line bundles. In this direction, we apply the main theorem of this section to construct interesting examples of abelian Galois covers admitting \rul sheaves.
We also show that this property of $M_i$ is stronger than $M_i$ being a 
globally generated line bundle. We also give an example showing these two properties are not equivalent. We also make some observations concerning the possible ranks of  \rul sheaves.  Moreover, we describe a relation between finite dimensional representations of certain generalized Clifford algebras and the existence of \rul sheaves.

The following proposition gives a sufficient condition for the existence of a rank $1$ \rul coherent sheaf on a double 
cover. The proof of the following proposition is simple and shows the idea of the proof for the general case.

\begin{proposition}\label{algstruc}
Take a line bundle $L\, \longrightarrow\, Y$ on a smooth projective variety $Y$ and a nonzero section
$s \,\in\, H^0(Y,\, L^{\otimes 2})$ such that
$s\,=\,a \otimes b$ for $a,\, b\, \in\, H^0(Y,\, L)$. Let $\pi\,:\, X \,\longrightarrow \, Y$ be the
corresponding double covering branched over the divisor of $s$. Then there exists a rank $1$ \rul coherent sheaf on
$X$ \wrt $\pi$.
\end{proposition}

\begin{proof}
We will use the method of matrix factorization. The method of matrix factorizations in the context
of \ul bundle has been explored in much generality in \cite{parameswaran2024ulrich}. Recall that the cyclic covering $X$ is
constructed as a closed subscheme of the total space $\text{Spec}(\text{Sym}(L^{-1}))$ of $L$ given by $(t^2-\theta^*s)$.
We denote the total space projection by $\theta\,:\, Z \,\longrightarrow\, Y$. Consider the projective compactification
\begin{equation}\label{vp}
\bP\ =\ \text{proj}(Sym(\cO_Y \oplus L^{-1}))\ \stackrel{\varpi}{\longrightarrow}\ Y .
\end{equation}
Take $e \,\in\, H^0(\bP,\, \cO_{\bP}(1))$ which is non-vanishing along $Z$, thus giving a trivialization
$\cO_Z \,\cong\, {\cO_{\bP}(1)}\big\vert_Z$. Let
$T$ be the section of $\varpi^*L \otimes \cO_{\bP}(1)$, where $\varpi$ is
the projection in \eqref{vp}, which restricts to $t \,\in\, H^0(Z,\, \theta^*L)$.
Consider the sections
$\varpi^*a \otimes e$ and $\varpi^*b \otimes e$ of $\varpi^*L \otimes \cO_{\bP}(1)$. Take the matrix
				$B\,=\,\begin{bmatrix}
					T & \varpi^*a \otimes e \\
					\varpi^*b \otimes e & T 
				\end{bmatrix}$
with coefficients in $H^0(\bP,\, \varpi^*L \otimes \cO_{\bP}(1))$. We have the following \ses
				\[\begin{tikzcd}
	0 & {\mathcal{O}_{\bP}(-1)^{\oplus 2}} \otimes \varpi^*L^{-1} && {\mathcal{O}_{\bP}^{\oplus 2}} & E & 0
					\arrow[from=1-1, to=1-2]
					\arrow["{\times B}", from=1-2, to=1-4]
					\arrow[from=1-4, to=1-5]
					\arrow[from=1-5, to=1-6]
				\end{tikzcd}\]
Then the cokernel sheaf $E$ is supported on the double covering $X$ given by the determinant of $B$. Since
$\varpi_* \cO_{\bP}(-1)\,=\,0$ and $R^1\varpi_*\cO_{\bP}(-1)\,=\,0$, applying $\varpi_*$
we get that $\pi_*E \,\cong\, \cO_Y^{\oplus 2}$. Thus, $E$ is a rank $1$ \rul coherent sheaf on $X$ \wrt $\pi$.
\end{proof}

\begin{remark}
	Suppose that the sections $a$ and $b$ have no common zero. Then the cokernel $E$ in the above proposition is a line bundle.
	 Locally on $Y$, we may write $X=\spec\, R[t]/(t^2-ab)$ and $B\,=\,\begin{bmatrix} t&a\\ b&t\end{bmatrix}$.
	On $X$, we have $\det B\,=\,t^2-ab\,=\,0$, so $B$ has rank at most one. Since $a$ and
	$b$ do not vanish simultaneously, at every point in $X$ at least one entry of
	$B$ is a unit in the corresponding local ring. Hence $B$ has constant rank one at every point in $X$. Moreover, since $\det(B)\,=\,0$, elementary row and column operations put it in the form $\begin{bmatrix}1&0\\0&0\end{bmatrix}.$
	Therefore, $\coker(B)$ is locally free of rank one. 
\end{remark}

The following is the main theorem of this section. When $\pi$ is a standard abelian covering, the existence of a \rul sheaf follows as a corollary of this theorem. 

\begin{theorem}\label{sufficient}
Let $\pi\,:\, X \,\longrightarrow\, Y$ be a standard cyclic covering map of degree $d$
between projective varieties. Assume that the branch divisor $s \,\in
\, H^0(Y,\, L^{\otimes d})$ is in the image of the $d$-fold multiplication map
$$H^0(Y,\, L)^{\otimes d} \ \longrightarrow\ H^0(Y,\, L^{\otimes d}).$$
Then there exists a \rul coherent sheaf $E$ on $X$ \wrt $\pi$.
\end{theorem}
		
\begin{proof}
{}From the assumption it follows that the branch locus $s\, \in \,H^0(Y,\, L^{\otimes d})$ is of the form
$$\sum_i a^1_i \otimes a^2_i \otimes \cdots \otimes a^d_i ,$$
where $a^j_i \,\in\, H^0(Y,\, L)$ for all $1 \,\le\, j \,\le\, d$, and all $i$. As
before, the cyclic covering $X$ is constructed as a closed subscheme of the total space
$\text{Spec}(\text{Sym}(L^{-1}))$ of $L$ given by $(t^d-\theta^*s)$. Denote the total space by $Z$, with the natural projection $\theta\,:\,Z \,\longrightarrow\,Y$, and let $t \,\in\,H^0(Z,\,\theta^*L)$ be the tautological section.
Consider the projective compactification $\bP\,=\,\text{Proj}({\rm Sym}(\cO_Y \oplus L^{-1}))$. 
We have the following diagram

\[\begin{tikzcd}
	X && Z && {\mathbb{P}} \\
	&& Y
	\arrow["{\text{closed}}", hook, from=1-1, to=1-3]
	\arrow["\pi", from=1-1, to=2-3]
	\arrow["{\text{open}}", hook, from=1-3, to=1-5]
	\arrow["\theta", from=1-3, to=2-3]
	\arrow["\varpi", from=1-5, to=2-3]
\end{tikzcd}\]

Take
$e \,\in\, H^0(\bP,\, \cO_{\bP}(1))$ which does not vanish on $Z$, thus giving a trivialization
$\cO_Z \,\cong\, {\cO_{\bP}(1)}\big\vert_{Z}$. Consider a matrix factorization of
$$T^d-\sum_i (\varpi^*a^1_i \otimes e) \otimes (\varpi^*a^2_i \otimes e) \otimes \cdots \otimes (\varpi^*a^d_i \otimes e).$$
The above expression is a section in $H^0(\bP,\, \varpi^*L^{\otimes d} \otimes \cO_{\bP}(d))$. The existence of a
matrix factorization given by $$B_1 \cdot B_2\cdot \ldots\cdot B_d\ =\
\big(T^d-\sum_i (\varpi^*a^1_i \otimes e) \otimes (\varpi^*a^2_i \otimes e) \otimes \cdots \otimes (\varpi^*a^d_i \otimes e) \big)
\times \id_{m \times m},$$
where $B_k$ are square matrices of size $m \times m$ with entries in $H^0(\bP, \, \varpi^*L \otimes \cO_{\bP}(1))$ is proved as follows. All the sections $T$ and $\delta_{ij}\,=\,\varpi^*(a^j_i) \otimes e$ belongs to $H^0(\bP, \varpi^*L \otimes \cO_{\bP}(1))$. Thus, any homogeneous degree $d$ polynomial in formal variables $x_0, x_{ij}$, evaluated at $x_0=T$ and $x_{ij}=\delta_{ij}$, gives a global section in $H^0(\bP, \varpi^*L^{\otimes d} \otimes \cO_{\bP}(d))$. We consider the homogeneous polynomial
$$G(x_0, x_{ij})\,=\,x_0^d-\sum_{i}x_{i1}x_{i2} \cdots x_{id}.$$

By \cite[Lemma 1.5]{HERZOG1991187}, there exist matrices $\beta_1, \beta_2, \cdots \beta_d$ whose entries are linear forms in variables $x_0, x_{ij}$ such that
$$\beta_1 \cdot \beta_2 \cdot \cdots \beta_d\,=\,G(x_0, x_{ij}) \cdot \id_{m \times m},$$
where the $\beta_k$ are matrices of size $m \times m$; see also \cite[Proposition 2.3]{parameswaran2024ulrich}.
Substituting $x_0=T$ and $x_{ij}=\delta_{ij}$, and denoting the resulting matrices by $B_k$, we obtain matrices $B_k$ with entries in $H^0(\bP, \varpi^*L \otimes \cO_{\bP}(1))$ such that
\begin{equation} \label{exmatfac}
	B_1 \cdot B_2 \cdots B_d\,=\,G(T, \delta_{ij}) \cdot \id_{m \times m}\,=\,\big(T^d-\sum_{i} (\varpi^*(a^1_i) \otimes e) \cdot (\varpi^*(a^2_i) \otimes e) \cdots (\varpi^*(a^d_i) \otimes e)\big)\, \cdot \id_{m \times m}.
\end{equation}
Then the cokernel for a  \ses given by the multiplication of $B_k$ will give \rul coherent sheaf on $X$ as follows:
\begin{equation}
\label{cokerseq}
\begin{tikzcd}
0 & {\mathcal{O}_{\bP}(-1)^{\oplus m}} \otimes \varpi^*L^{-1} && {\mathcal{O}_{\bP}^{\oplus m}} & \widetilde{E} & 0
\arrow[from=1-1, to=1-2]
\arrow["{\times B_k}", from=1-2, to=1-4]
\arrow[from=1-4, to=1-5]
\arrow[from=1-5, to=1-6]
\end{tikzcd}
\end{equation}
where $m$ is the size of the matrices $B_k$. Then the cokernel sheaf $\widetilde{E}$ is supported along the
determinant of $B_k$; this determinant can be shown to be $\big(T^d-\sum_i ((\varpi^*a^1_i \otimes e)
\otimes (\varpi^*a^2_i \otimes e) \otimes \cdots \otimes (\varpi^*a^d_i \otimes e))\big)^{\frac{m}{d}}$.  The cyclic covering $j\,:\,X \longhookrightarrow \bP$ is given by $T^d-\sum_i (\varpi^*a^1_i \otimes e) \otimes (\varpi^*a^2_i \otimes e) \otimes \cdots \otimes (\varpi^*a^d_i \otimes e)$. Therefore, the exists a coherent sheaf $E$ on $X$ such that $j_* E \cong \widetilde{E}$ (for more details see Remark \ref{support}).

Since $\varpi_* \cO_{\bP}(-1)\,=\,0$ and $R^1\varpi_*\cO_{\bP}(-1)\,=\,0$, applying $\varpi_*$ to
\eqref{cokerseq} we get that $\varpi_*\widetilde{E} \,\cong\, \cO_Y^{\oplus m}$. Thus, $\pi_* E\,=\,(\varpi \circ j)_* E \,=\,\varpi_* (j_* E)\,\cong\, \varpi_* \widetilde{E} \,\cong\, \cO_Y^{\oplus m}$. This implies that $E$ is a \rul
coherent sheaf of rank $\frac{m}{d}$. 
In this case, there will be no control over the rank of the \rul coherent sheaf. The rank will depend on the size of the matrix factorization $m$, or in other words, on the expression of $s$.
\end{proof}

\begin{remark}\label{support}
		A priori, the cokernel $\widetilde{E}$ may be supported on a nonreduced scheme $\widetilde{X}$ whose reduced induced subscheme is $X$.
		Let $I_{X}$ be the ideal sheaf of of the hypersurface $j\,:\,X \longhookrightarrow \bP$ defined by $F\,:\,=T^d-\sum_i (\varpi^*a^1_i \otimes e) \otimes (\varpi^*a^2_i \otimes e) \otimes \cdots \otimes (\varpi^*a^d_i \otimes e)$. To show that $\widetilde{E}$ is supported on $X$, it is enough to prove that $I_X \cdot \widetilde{E}=0$. It is known that all cyclic rotations of the product appearing in the matrix factorization \eqref{exmatfac} are equal to $F \cdot \id_{m \times m}$. That is, for every $1 \le k \le d$, $B_k \cdot B_{k+1} \cdots B_d \cdot B_1 \cdots B_{k-1}=F \cdot \id_{m \times m}$. Consider the diagram
	\[
	\begin{tikzcd}[column sep=large]
		{\mathcal{O}_{\mathbb{P}}(-d)^{\oplus m}\otimes \varpi^*L^{-d}}
		\arrow[r, "{\times (B_{k+1}\cdots B_{k-1})}"]
		&[5em]
		{\mathcal{O}_{\mathbb{P}}(-1)^{\oplus m}\otimes \varpi^*L^{-1}}
		\arrow[r, "{\times B_k}"]
		&
		{\mathcal{O}_{\mathbb{P}}^{\oplus m}}
	\end{tikzcd}
	\]
	Since $B_k \cdot B_{k+1} \cdots B_d \cdot B_1 \cdots B_{k-1}\,=\,F \cdot \id_{m \times m}$,  multiplication by $F$ is zero $\widetilde{E}\,=\,\coker (B_k)$.
	 Thus $I_X \cdot \widetilde{E}\,=\,0$. The category of $\cO_X$ modules is equivalent to  the category of $\cO_{\bP}$ modules annihilated by $I_X$.  Therefore there exists a coherent sheaf $E$ on $X$ such that $j_* E \,\cong \, \widetilde{E}$.
	\end{remark}

\begin{remark}
	Let $Y$ be Cohen-Macaulay. Since $\bP$ is a projective bundle over $Y$, $\bP$ is also Cohen-Macaulay. Let $Q \in X$. By the definition sequence \eqref{cokerseq}, the sheaf  $\widetilde{E}$, regarded as an $\cO_{\bP}$ module has a locally free resolution of length $1$. Hence $\text{pd}_{\cO_{\bP, Q} }(\widetilde{E}_Q)\,=\,1$. By Auslander-Buchsbaum formula, we obtain $\text{depth}_{\cO_{\bP, Q} }(\widetilde{E}_Q)\,=\,\dim \cO_{\bP, Q} -1\,=\,\dim \cO_{X, Q}$, where the last equality follows since $X \subset \bP$ is a hypersurface. By Remark \ref{support}, there exists a coherent sheaf $E$ on $X$ such that $j_* E \cong \widetilde{E}$. Moreover, for the closed immersion $j\,:\,X \longhookrightarrow \bP$, we have
	$$\,\text{depth}_{\cO_{X, Q} }(E_Q)\,=\text{depth}_{\cO_{\bP, Q} }((j_*E)_Q)\,=\,\text{depth}_{\cO_{\bP, Q} }(\widetilde{E}_Q)\,=\,\dim \cO_{X, Q}.$$
	Therefore $E$ is a maximal Cohen-Macaulay relatively Ulrich sheaf. 
\end{remark}

\begin{corollary} \label{standardsuff}
	Let $\pi\,:\, X \,\longrightarrow\, Y$ be a standard abelian Galois covering of degree $d\,=\,(d_1,\,
	d_2,\, \cdots,\, d_l)$ as in Definition \ref{standardabelian}. Assume that the branch divisor $B_i$ is in the
	image of the $d_i$-fold multiplication map
	$$H^0(X_i,\, N_i)^{\otimes d_i} \ \longrightarrow\ H^0(X_i,\, N_i^{\otimes d_i})$$
	for all $1 \,\le\, i \,\le\, l$. Then there exists a \rul coherent sheaf on $X$ \wrt $\pi$. 	
\end{corollary}

\begin{proof}
	Recall that $\pi$ is the composition of standard cyclic coverings
	\[\begin{tikzcd}
		{X=X_0} & {X_1} & {X_2} & \ldots & {X_{l-1}} & {X_{l}=Y}
		\arrow["{\pi_1}", from=1-1, to=1-2]
		\arrow["{\pi_2}", from=1-2, to=1-3]
		\arrow["{\pi_3}", from=1-3, to=1-4]
		\arrow[from=1-4, to=1-5]
		\arrow["{\pi_l}", from=1-5, to=1-6]
	\end{tikzcd}\]	
where $\pi_i\,:\,X_{i-1} \,\longrightarrow\, X_{i}$ is a standard cyclic covering of degree $d_i$.
We consider the branch divisor $B_{i}
\,\in\, |N_i^{\otimes d_i}|$ on $X_i$. Then $X_{i-1}$ is
the associated cyclic covering.
Thus, by Theorem \ref{sufficient}, there exists a \rul coherent sheaf $E_i$ on $X_{i-1}$ \wrt $\pi_i$ for each
$1 \,\le\, i \,\le\, l$.

We can construct a \rul sheaf $E$ on $X$ by induction in the following way.
Assume that $F$ is \rul on $X_i$ \wrt the map $ \pi_l \circ \pi_{l-1} \circ \ldots \circ\pi_{i+1}
\,:\, X_i \,\longrightarrow\, Y$. From the construction, we have a \rul sheaf $E_i$ on $X_{i-1}$ \wrt 
$\pi_i\,:\,X_{i-1} \,\longrightarrow \,X_i$. Using projection formula, $$ \pi_i^*F \otimes E_i$$
is a \rul sheaf on $X_{i-1}$ \wrt the map
$ \pi_l \circ \pi_{l-1} \circ \ldots \circ\pi_{i+1} \circ \pi_i\,:
\, X_{i-1} \,\longrightarrow \,Y$. Hence, by induction, there exists a \rul sheaf $E$ on $X$ \wrt $\pi$.
\end{proof}

\begin{corollary} \label{admissiblecor}
Let $\pi\,:\, X \,\longrightarrow\, Y$ be an admissible  abelian Galois covering of degree $d\,=\,(d_1,\,
	d_2,\, \cdots,\, d_l)$ as in Definition \ref{admissibleab}. Assume that the branch divisor $B'_i$ is in the
	image of the $d_i$-fold multiplication map
	$$H^0(Y,\, M_i)^{\otimes d_i} \ \longrightarrow\ H^0(Y,\, M_i^{\otimes d_i})$$
	for all $1 \,\le\, i \,\le\, l$. Then there exists a \rul coherent sheaf on $X$ \wrt $\pi$. 	
\end{corollary}

\begin{proof}
 	Since every admissible abelian Galois covering is a standard abelian cover, we have $\pi$ is the composition of cyclic coverings
 \[\begin{tikzcd}
 	{X=X_0} & {X_1} & {X_2} & \ldots & {X_{l-1}} & {X_{l}=Y}
 	\arrow["{\pi_1}", from=1-1, to=1-2]
 	\arrow["{\pi_2}", from=1-2, to=1-3]
 	\arrow["{\pi_3}", from=1-3, to=1-4]
 	\arrow[from=1-4, to=1-5]
 	\arrow["{\pi_l}", from=1-5, to=1-6]
 \end{tikzcd}\]	
 where $\pi_i\,:\,X_{i-1} \,\longrightarrow\, X_{i}$ is a standard cyclic covering of degree $d_i$. We recall the map $\pi_{i+1}'\,=\,\pi_{i+1} \circ \pi_{i-1} \circ \ldots \circ \pi_l\,:\,X_i \longrightarrow Y$ (see \eqref{keydiag10}).
 We consider the branch divisor $B_i\,:=\,\,\pi_{i+1}'^*(B'_{i})
 \,\in\, |\pi_{i+1}'^*(M_i^{\otimes d_i})|$ on $X_i$. Then $X_{i-1}$ is
 the associated cyclic covering for the data $(B_i, \pi_{i+1}'^* M_i)$.
 Since $B'_i$ is in the image of the $d_i$-fold multiplication map
 $$H^0(Y,\, M_i)^{\otimes d_i}\ \longrightarrow\ H^0(Y,\, M^{\otimes d_i}_i),$$
 it follows that $B_i=\pi_{i+1}'^*(B_{i}')$ is in the image of the
 $d_i$-fold multiplication map
 $$H^0\big(X_i, \, \pi_{i+1}'^* M_i\big)^{\otimes d_i}\ \longrightarrow
 \ H^0\big(X_i, \, (\pi_{i+1}'^*M_i)^{\otimes d_i}\big).$$
 Thus, by Corollary \ref{standardsuff}, there exists a \rul sheaf $E$ on $X$ \wrt $\pi$.
\end{proof}

Proposition \ref{algstruc} is a special case of a more general phenomenon. We generalize it below. Proposition \ref{minrank1} and Corollary \ref{minrank2}  may also be viewed as special cases of the broader Question \ref{Questrank}. 

\begin{proposition} \label{minrank1}
	Let $\pi\,:\,X \longrightarrow Y$ be a standard cyclic cover of degree $d$ of projective varieties. Let $L$ be a line bundle $Y$, and $s \,\in\, H^0(Y,\, L^{\otimes d})$ be a nonzero section defining the branch divisor $B$ of $\pi$.
	Assume there exists a $d \times d$ matrix $C$ with entries in $H^0(Y, L)$ such that $\det (x \id_{d \times d}-C)=x^d-s$ as a formal homogeneous identity, where $x$ is formal variable and multiplication $x \cdot \alpha$, for $\alpha \in H^0(Y, L)$, is understood formally. Then there exists a rank $1$ relatively Ulrich sheaf on $X$ \wrt $\pi$.
	\end{proposition}

\begin{proof}
	 Recall that the cyclic covering $X$ is
	constructed as a closed subscheme of the total space $\text{Spec}(\text{Sym}(L^{-1}))$ of $L$ given by $(t^d-\theta^*s)$.
	We denote the total space by $\theta: Z \longrightarrow Y$. Consider the projective compactification
	\begin{equation}
		\bP\ =\ \text{proj}(Sym(\cO_Y \oplus L^{-1}))\ \stackrel{\varpi}{\longrightarrow}\ Y .
	\end{equation}
	Take $e \,\in\, H^0(\bP,\, \cO_{\bP}(1))$ which is non-vanishing along $Z$, thus giving a trivialization
	$\cO_Z \,\cong\, {\cO_{\bP}(1)}\big\vert_Z$. Let
	$T$ be the section of $\varpi^*L \otimes \cO_{\bP}(1)$, where $\varpi$ is
	the projection in \eqref{vp}, which restricts to $t \,\in\, H^0(Z,\, \theta^*L)$.
	Then $\varpi^* s \otimes e^d \in H^0(\bP,\,\varpi^*L^{\otimes d} \otimes \cO_{\bP}(d))$. We consider the matrix $B=e \cdot \varpi^*(C)$, so that the $ij$-th entry is given by $B_{ij}\,=\,(e \cdot \varpi^*(C))_{ij}\,=\,e \otimes \varpi^*(C_{ij})$. Thus, $B$ is a $d \times d$ matrix with entries in $H^0(\bP, \varpi^*L \otimes \cO_{\bP}(1))$. By the assumption on $C$, we have $\det (ex \id_{d \times d}-B)=(ex)^d-\varpi^*s \otimes e^d$. Substituting $ex\,=\,T$, we  obtain  $\det (T\id_{d \times d}-B)=T^d-\varpi^*s \otimes e^d$. Put $S\,:\,=T\id_{d \times d}-B$. Thus, $S$ is a $d \times d$ matrix with entries in $H^0(\bP, \varpi^*L \otimes \cO_{\bP}(1))$. Consider the following \ses
	\[\begin{tikzcd}
		0 & {\mathcal{O}_{\bP}(-1)^{\oplus d}} \otimes \varpi^*L^{-1} && {\mathcal{O}_{\bP}^{\oplus d}} & E & 0
		\arrow[from=1-1, to=1-2]
		\arrow["{\times S}", from=1-2, to=1-4]
		\arrow[from=1-4, to=1-5]
		\arrow[from=1-5, to=1-6]
	\end{tikzcd}\]
	Then the cokernel sheaf $E$ is supported on the cyclic covering $X=Z(T^d-\varpi^*s \otimes e^d)$ given by the determinant of $S$. Since
	$\varpi_* \cO_{\bP}(-1)\,=\,0$ and $R^1\varpi_*\cO_{\bP}(-1)\,=\,0$, applying $\varpi_*$
	we get that $\pi_*E \,\cong\, \cO_Y^{\oplus d}$.
	Therefore $E$ is a relatively Ulrich coherent sheaf of rank $1$.
\end{proof}

As a corollary, we generalize Proposition  \ref{algstruc}.
\begin{corollary} \label{minrank2}
	Let $L\, \longrightarrow\, Y$ be a line bundle on a  projective variety $Y$ and a nonzero section
	$s \,\in\, H^0(Y,\, L^{\otimes d})$ such that
	$s\,=\,a_1 \otimes a_2 \otimes \cdots \otimes a_d$ with $a_i \in\, H^0(Y,\, L)$ for each $1 \le i \le d$. Let $\pi\,:\, X \,\longrightarrow \, Y$ be the
	corresponding standard cyclic covering of degree $d$ branched over the divisor of $s$. Then there exists a \rul coherent sheaf of rank $1$ on $X$ \wrt $\pi$. 
\end{corollary}

\begin{proof}
	We consider the following matrix
	\footnotesize
	\[
	C=
	\begin{bmatrix}
		0 & a_{1} & 0 & \cdots & 0 \\
		0 & 0 & a_{2} & \cdots & 0 \\
		\vdots & \vdots & \ddots & \ddots & \vdots \\
		0 & 0 & \cdots & 0 & a_{d-1} \\
		a_{d} & 0 & \cdots & 0 & 0
	\end{bmatrix}.
	\]
	\normalsize
	Then we have $C^{d}\,=\,a_1 \cdot a_2 \cdot \cdots \cdot a_d \id_{d \times d}\,=\,s \cdot \id_{d \times d}$, where the product $a_1 \cdot a_2 \cdots \cdots a_d$ is identified with the section $s \in H^0(Y,\,L^{\otimes d})$. For a formal variable $x$, we also have the formal identity $\det (x \id_{d \times d}-C)\,=\,x^d-\,a_1 \cdot a_2 \cdot \cdots \cdot a_d =x^d-s$. Therefore by Proposition \ref{minrank1}, there exists a \rul coherent sheaf $E$ of rank $1$ on $X$ \wrt $\pi$. 
\end{proof}

In the following proposition, we prove that, branch divisors defined by sections of suitable line bundles in 
a complete intersection smooth, projective variety satisfies our sufficient criterion.

\begin{proposition} \label{ciexamples}
Let $Z \,\hookrightarrow\, \bP^N$ be a complete intersection subvariety defined by vanishing locus of polynomials of degrees
$a_1,\, a_2,\, \cdots,\, a_k$. Then for any $n \ge 0$, for the line bundle $L\,=\,\cO_{\bP^N}(n)\big\vert_Z$ on $Z$,
the $m$-fold multiplication map $$H^0(Z,\, L)^{\otimes m}\ \longrightarrow\ H^0(Z,\, L^{\otimes m})$$
is surjective.
\end{proposition}

\begin{proof}
Consider the following natural commutative diagram
\footnotesize
  \begin{equation} \label{commutative}
  \begin{tikzcd}
	{H^0(\bP^N, \, \mathcal{O}_{\mathbb{P}^N}(n))^{\otimes m}} &&& {H^0(\bP^N, \, \mathcal{O}_{\mathbb{P}^N}(m \cdot n))} \\
	\\
	{H^0(Z, \, L)^{\otimes m}} &&& {H^0(Z, \, L^{\otimes m})}
	\arrow["{\text{multiplication}}", from=1-1, to=1-4]
	\arrow["{\text{restriction}}"', from=1-1, to=3-1]
	\arrow["{\text{restriction}}", from=1-4, to=3-4]
	\arrow["{\text{multiplication}}"', from=3-1, to=3-4]
\end{tikzcd}
\end{equation}
\normalsize
On $\bP^N$, we have the $m$-fold multiplication map
$${H^0(\mathbb{P}^N,\, \mathcal{O}_{\mathbb{P}^N}(n))^{\otimes m}}\ \longrightarrow\
{H^0(\mathbb{P}^N,\, \mathcal{O}_{\mathbb{P}^N}(m \cdot n))}$$
is surjective. From the commutativity of \eqref{commutative}, it is enough to show that the restriction map
$$H^0(\mathbb{P}^N,\, \mathcal{O}_{\mathbb{P}^N}(m \cdot n))\ \longrightarrow\ {H^0(Z,\, L^{\otimes m})}$$
is surjective. This will follow using induction. We denote the intermediate subvarieties by $ \bP^N \, \supset Z_1
\, \supset \, Z_2 \, \supset \, Z_3 \, \supset \, \cdots \,  \, \supset\, Z_{k-1} \,\supset\, Z_k=Z$ of
$\bP^N$ successively defined by the vanishing locus of polynomials of degrees $a_1,\, a_2, \, \cdots,\, a_k$ respectively.
Consider the \ses
$$0 \,\longrightarrow\, \cO_{\bP^N}(-a_i)\big\vert_{Z_{i-1}}\, \longrightarrow\, \cO_{Z_{i-1}}
\, \longrightarrow\, \cO_{Z_i} \,\longrightarrow\, 0.$$
Tensoring it with $\cO_{\bP^N}(m \cdot n)$, and taking the \les in cohomology, we get the map 
\begin{equation} \label{restriction}
H^0(Z_{i-1},\, \cO_{\bP^N}(m \cdot n)\big\vert_{Z_{i-1}})\ \longrightarrow\ H^0(Z_{i},\,
\cO_{\bP^N}(m \cdot n)\big\vert_{Z_{i}})
\end{equation}
is surjective, provided that $H^1(Z_{i-1},\, \cO_{\bP^N}((m \cdot n)-a_i)\big\vert_{Z_{i-1}})\,=\,0$. Thus  it is enough to show that $H^1(Z_i, \, \cO_{\bP^N}(l)\big\vert_{Z_{i}})\,=\,0$ for all
$1 \,\le\, i \,\le\, (k-1)$ and for all $l \,\in\, \bZ$. We observe that $\dim\, (Z_i) \,\ge\, 2$ for all
$1 \,\le\, i \,\le\, (k-1)$. Using arguments
similar to the proof in \cite[Theorem 4.2]{parameswaran2024ulrich}, we can prove
that $H^1(Z_i, \, \cO_{\bP^N}(l)\big\vert_{Z_i})\,=\,0$ for all $1 \,\le\, i \,\le\, (k-1)$ and for all $l \,\in\, \bZ$. Thus
the restriction map in \eqref{restriction} is surjective for every $1 \,\le\, i \,\le\, k$. Taking the
composition of these restriction maps, we get 
$$H^0(\bP^N, \, \mathcal{O}_{\mathbb{P}^N}(m \cdot n))\ \longrightarrow\ {H^0(Z, \, L^{\otimes m})}$$ is surjective.
\end{proof}

Using Proposition \ref{ciexamples}, we can construct numerous examples of abelian coverings admitting \rul 
bundles. Let $L_i$ be line bundle $\cO_{\bP^N}(\alpha_i)\big\vert_{Z}$ on $Z$ for all $1 \,\le\, i \,\le\, l$. We consider 
the branch divisors $s_i \,\in\, H^0(Z,\, L_i^{\otimes d_i})$ for every $1\,\le\, i \,\le\, l$. Let $\pi\,:\, X 
\, \longrightarrow\, Z$ be the associated admissible abelian covering (see \eqref{a2}). Then by Proposition \ref{ciexamples}, 
the sufficient criterion in Corollary \ref{admissiblecor} for $s_i$ is satisfied for all $1 \,\le\, i \,\le\, l$.
Consequently, the following is obtained:

\begin{corollary} \label{comintcor}
There exists a \rul coherent sheaf on $X$ \wrt $\pi$.
\end{corollary}
	
The sufficient criterion in Theorem \ref{sufficient} is closely related to the notion of normal generation of 
the line bundle $L$. Thus, when $Y$ is normal and $L$ is very ample, the criterion is related to the 
projective normality of the embedding of $Y$ in $\bP(H^0(Y,\, L))$.

Recall that a line bundle $L$ on a projective variety $Y$ is normally generated if the natural multiplication 
map $Sym^d H^0(Y,\, L)\, \longrightarrow\, H^0(Y,\, L^{\otimes d})$ is surjective for all $d \ge 1$. Assume 
that $L$ is very ample and defines an embedding $Y\, \longrightarrow\, \bP(H^0(Y,\, L))$. If $Y$ is normal, 
then $L$ is normally generated if and only if this embedding is projectively normal.

The following examples are deduced as applications of Corollary \ref{admissiblecor}.

\begin{example} \label{ngex1}
Let $C$ be a smooth projective curve of genus $g$, and let $L$ be a line bundle on $C$. By a classical 
result, if $\deg (L) \,\ge\, 2g+1$, then $L$ is normally generated \cite[\S~1.8.D, p. 116]{Lazarsfeld17}. Fix a 
tuple $d\,=\,(d_1,\, d_1,\, \cdots,\, d_l)$, where each $d_i \,\ge\, 2$ is an integer. Let $L_i$ be line bundles
on $C$ 
such that $\deg (L_i) \,\ge \,2g+1$. Therefore, every admissible abelian Galois cover 
$\pi\,:\,D\,\longrightarrow\, C$ associated with the data $$\{(d_i, B_i',\, L_i)\,\ \big\vert\,\ B_i'\ \in\ |L_i^{\otimes 
d_i}|\}$$ admits a relatively Ulrich sheaf \wrt $\pi$. In particular, when $l\,=\,1$, every smooth ramified 
cyclic cover of $C$ admits a relatively Ulrich vector bundle.
\end{example}	

\begin{remark}
In higher dimensions, similar examples of relatively Ulrich sheaves can be constructed using the
Castelnuavo-Mumford regularity theorem \cite[Theorem 1.8.5]{Lazarsfeld17}. 
\end{remark}

\begin{example} \label{ngex2}
	Let $Y$ be a smooth projective toric surface. For each $1 \,\le\, i \,\le\, l$, let $M_i$ be an ample line bundle on $Y$. Then it is known that $M_i$ is also globally generated. By \cite[Theorem 1]{Fakhruddin02}, the multiplication map 
	$$H^0(Y, \,M_i)^{\otimes d_i} \,\longrightarrow\, H^0(Y,\,M_i^{\otimes d_i})$$
	is surjective. Hence, every admissible abelian cover $\pi\,:\,X \longrightarrow Y$ constructed from the data $(B'_i,\,M_i)$ with $B'_i\, \in\, |M_i^{\otimes d_i}|$ admits a relatively Ulrich sheaf $E$ on $X$ \wrt $\pi$. Moreover, if $X$ is smooth, then $E$ is a \rul bundle on $X$ \wrt $\pi$.
\end{example}

	In the following remark, we describe an equivalence between existence of relatively Ulrich sheaves and representations of certain generalized Clifford algebras. The formulation is motivated by the generalized Clifford algebra framework used in \cite{CKM2011}.
	
\begin{remark}[{Connection to representations of generalized Clifford algebra}]\label{gencliffrel}
Let $L$ be a line bundle on a projective variety $Y$. Take a nonzero section $s\,\in\, H^0(Y,\, L^{\otimes d})$
for some integer $d \,\geq\, 2$. Let $\pi\,:\,X\, \longrightarrow\, Y$ be the standard cyclic cover of degree
$d$ associated with the data $(s,\, L)$.
Denote $V\, =\, H^0(Y,\, L)$ and $W\,=\, H^0(Y,\, L^{\otimes d})$, and consider the  multiplication map
$$\mu_d\ :\ V^{\otimes d}\ \longrightarrow \ W.$$
Let $\mathbf {x} \,\in\, V^{\vee}  \otimes V$ be the element corresponding to
$\id_{V} \, \in\,  \End(V,\,V)$. Denote the tensor algebra of $V^{\vee}$ by $T(V^{\vee})$. We define
		\begin{equation}
			C(Y,L,s) \,=\,T(V^\vee)\Big/ \left\langle (\id \otimes\lambda) \left( (\id \otimes\mu_d)(\mathbf x^d)-1\otimes s \right) \,:\,\lambda\in W^\vee
			\right\rangle .
		\end{equation}
		Here $\mathbf{x}^d \in T(V^{\vee}) \otimes V^{\otimes  d}$. Indeed, if $v_1, v_2, \cdots, v_n$ is a basis of $V$ and $x_1, x_2, \cdots, x_n$ is the dual basis of $V^\vee$, then $\mathbf{x}\,=\,\sum_{i} x_i \otimes v_i$. Hence $\mathbf{x}^d\,=\,\big(\sum_{i} x_i \otimes v_i\big)^d\,=\,\sum_{i_1, i_2, \cdots, i_d} (x_{i_1} x_{i_2}\cdots x_{i_d}) \otimes (v_{i_1} \otimes v_{i_2} \otimes \cdots \otimes v_{i_d}) \in T(V^{\vee}) \otimes V^{\otimes  d}$. Here $(x_{i_1} x_{i_2}\cdots x_{i_d})$ denotes the noncommutative product in $T(V^{\vee})$, whereas $(v_{i_1} \otimes v_{i_2} \otimes \cdots \otimes v_{i_d})$ is the ordinary tensor product of sections. Thus $\big((\id \otimes\mu_d)(\mathbf x^d)-1\otimes s \big) \in T(V^{\vee}) \otimes W$.

		By the universal property of the tensor algebra, a representation $\rho\,:\,T(V^{\vee}) \longrightarrow \End_k(U)$, where $U$ is a finite dimensional $k$-vector space, is the same as a $k$-linear map $V^{\vee} \longrightarrow \End_k(U)$. Equivalently, such a representation determines an element $A_\rho \in \End_k(U) \otimes V\,=\,\End_k(U) \otimes H^0(Y, L)$. Now $A_{\rho}^d \in \End_k(U) \otimes V^{\otimes d}$ is formed using multiplication in $\End_k(U)$ and the ordered tensor product in $V$. We then consider $(\id \otimes \mu_d)(A_{\rho}^d) \in \End_k(U) \otimes W$. 
		
		The quotient $C(Y,L,s)$ imposes the relation $(\id \otimes\mu_d)(\mathbf x^d)\,=\,1\otimes s$. The representation $\rho$ sends the universal element $\mathbf{x}$ to $A_{\rho}$. Hence, the relation becomes $ (\id \otimes \mu_d)(A_{\rho}^d)\,=\,\id_{U} \otimes s$ in $\End_k(U) \otimes W$.
		
Equivalently, $A_{\rho}$ defines an $\mathcal O_Y$-linear map
		\begin{equation} \label{matrdatum}
			\varphi_{\rho}\ :\ U\otimes_k\mathcal O_Y\ \longrightarrow\ (U\otimes_k\mathcal O_Y) \otimes_{\mathcal O_Y} L ,
		\end{equation}
		satisfying $\varphi_{\rho}^d\,=\,s\cdot\id_{U\otimes_k\mathcal O_Y}$ as a morphism $U\otimes_k\mathcal O_Y \longrightarrow (U\otimes_k\mathcal O_Y) \otimes L^{\otimes d}$.
		
		We now explain the equivalence between a datum $\varphi_{\rho}$ (see \eqref{matrdatum}),
satisfying $\varphi_{\rho}^d\,=\,s\cdot\id_{U\otimes_k\mathcal O_Y}$
and the existence of a relatively Ulrich sheaf $E$ on $X$ \wrt $\pi$. A datum $\varphi_{\rho}$ \eqref{matrdatum}
implies that $s\,=\,\sum_i a_i^1\otimes a_i^2 \otimes \cdots\otimes a_i^d$. Thus, by the proof of
Theorem \ref{sufficient}, there exists a relatively Ulrich sheaf on $X$ \wrt $\pi$. Conversely, by
Proposition \ref{dfoldsurj}, the existence of a relatively Ulrich sheaf on $X$ \wrt $\pi$ induces a datum $\varphi_{\rho}$ satisfying the above property.
\end{remark}

\begin{remark}
Let $\pi\,:\, X\, \longrightarrow\, Y$ be a cyclic covering of smooth, projective varieties. We
assume that the branch divisor $s$ is in the image of the $d$-fold multiplication map
$H^0(Y,\, L^{\otimes d}) \,\longrightarrow\, H^0(Y,\, L)^{\otimes d}$. Then by Theorem \ref{sufficient}, there
exists a \rul bundle on $X$ \wrt $\pi$. Hence, by Proposition \ref{ggcyclic}, $L$ is generated by global
sections. In the following proposition, we will give a direct proof that the former criterion implies the latter. In fact, the proposition can be proved under weaker hypothesis: we do not assume that $X$ and $Y$ are smooth, but require the branch divisor $B$ to be smooth.
 \end{remark}

\begin{proposition} \label{dfoldgg}
Let $\pi\,:\, X \,\longrightarrow\, Y$ be a standard cyclic covering of projective varieties such that the
branch divisor $s$ is in the image of the $d$-fold multiplication map $$H^0(Y,\, L)^{\otimes d}
\ \longrightarrow\ H^0(Y,\, L^{\otimes d}).$$
Assume that $B$ is a smooth divisor. Then $L$ is generated by global sections. 
\end{proposition} 

\begin{proof}
	Since $s$ is in the image of the $d$-fold multiplication map, $s$ is of the form
	$$s\ =\ \sum_{i \in I} a^i_1 \otimes a^i_2 \otimes \ldots \otimes a^i_d.$$
	Then the sections $\{a^i_j\,\,\big\vert\,\, 1 \,\le\, j \,\le\, d,\, i \,\in\, I\}$ of $H^0(Y,\,
L)$ generates $L$ can be in the following way.\\
	
\textbf{\underline{Case 1}} Let $x \,\notin\, Z(s)$. Then at-least, one of the $a^i_j$ does not vanish at 
$x$, giving $0 \,\ne\, (a^i_j)_{|x} \,\in\, L_x$.\\
	
\textbf{\underline{Case 2}} Let $x \,\in\, Z(s)$. If for all $i,\, j$, the sections $a^i_j$ vanish at $x$, then 
the branch divisor $B\,=\,Z(s)$ will have a singularity at $x$, contradicting the smoothness of $B$. Thus at-least 
one of the $a^i_j$ is nonzero at $x$ generating $L_x$.
\end{proof}

Let $\pi\,:\, X \,\longrightarrow\, Y$ be a standard cyclic Galois covering of degree $d$ between smooth
projective varieties. The 
sufficient condition for the existence of \rul vector bundles, namely the condition
that the branch divisor in $H^0(Y,\, L^{\otimes d})$ is in 
the image of the $d$-fold multiplication map, is stronger than the condition
that $L$ is a globally generated line bundle 
on $Y$ (see Proposition \ref{dfoldgg}). In the following example, we construct  a cyclic covering such that the line bundle $L$ is globally 
generated, and  yet there are no \rul vector bundles. 

\begin{example} \label{ggdfoldexample}
 Let $L$ be a line bundle of degree two on an elliptic curve $D$. Then $h^0(D, \, L)\,=\,2$ by Riemann-Roch.
For a point $x\, \in\, D$, using the \ses
 $$0 \ \longrightarrow\ L(-x) \ \longrightarrow\ L\ \longrightarrow\ L_{|x}\ \longrightarrow\ 0,$$
it follows that the evaluation map $H^0(D,\, L) \,\longrightarrow\, L_{|x}$ is surjective. Thus $L$ is
globally generated. Consider the evaluation map
 \begin{equation} \label{ggseq}
 0\ \longrightarrow\ L^{-1}\ \longrightarrow\ H^0(D,\, L) \otimes_k \cO_D\ \longrightarrow\ L
\ \longrightarrow\ 0.
 \end{equation}
Applying the tensor product by $L$ to \eqref{ggseq}, and taking the \les of cohomologies, we get the
following exact sequence
$$0\, \longrightarrow\, H^0(D,\, \cO_D)\, \longrightarrow\, H^0(D,\, L) \otimes H^0(D,\, L)
\, \longrightarrow\, H^0(D,\, L^{\otimes 2})\, \longrightarrow\, H^1(D,\, \cO_D) \,\longrightarrow\, 0.$$ 
This shows that the natural map $H^0(D,\, L) \otimes H^0(D,\, L)\,\longrightarrow\, H^0(D,\, L^{\otimes 2})$
is not surjective. Thus we can construct a double cover $\pi\,:\,C \,\longrightarrow\, D$ branched over a
smooth divisor in $|H^0(D,\, L^{\otimes 2})|$ outside the image of $H^0(D,\, L) \otimes H^0(D,\,L)$. Note that $C$ is smooth because it corresponds to a divisor lying outside the image of $H^0(D,\, L) \otimes H^0(D,\,L)$. Then
there is no \rul vector bundle on $C$ \wrt $\pi$ by Proposition \ref{dfoldsurj}. It will be interesting
to give a direct proof that there is no \rul vector bundle on $C$ \wrt $\pi$.
\end{example}

We end the paper by posing the following question.
\begin{question} \label{Questrank}
We have provided a sufficient condition for the existence of a \rul coherent sheaves. It would be interesting to 
determine the minimal possible rank of \rul sheaves (or vector bundles when applicable) on a given class of varieties. In Corollary \ref{minrank2} it was shown that when the branch section is of the form $s\,=\,a_1 \otimes a_1 \otimes \cdots \otimes a_d$, there exists a \rul sheaf of rank $1$ on the associated cyclic covering. More generally, It would be interesting to relate specific forms of the branch divisor 
$s$ and the minimal rank of \rul sheaves on the associated cover. As a first case, one may consider
 cyclic covers of complete intersection curves.
\end{question}

The analogue question has been studied for Ulrich bundles in the case where the base $Y$ is a projective space. In particular, it was studied for double covers of $\bP^2$ in \cite[Lemma 4.2, Theorem 4.4]{MOHANKUMAR2025107946}, and was later generalized to double covers of $\bP^n$ in \cite[Theorem 1.2]{vacca25}.

 \subsection*{Acknowledgements}

The authors are very grateful to the referee for very helpful comments.
The first-named author is supported by a J. C. Bose Fellowship (JBR/2023/000003). The second-named author was 
supported by a postdoctoral fellowship at TIFR, Mumbai, where this work was initiated during a visit by the 
first-named author. The second-named author was subsequently supported by a Visiting Scientist position at 
ISI Bangalore, during which this work was completed. The second-named author is currently supported by an 
Institute postdoctoral fellowship at IISER Pune. The authors thank Roberto Vacca for bringing the references 
\cite{AK2026} and \cite{vacca25} to their attention.

\end{document}